\documentclass[12pt]{amsart}
\usepackage[utf8]{inputenc}
\usepackage[T1]{fontenc}
\usepackage[english]{babel}
\usepackage{graphicx}
\usepackage{amsthm}
\usepackage{amssymb}
\usepackage{stmaryrd}
\usepackage{csquotes}
\usepackage{amsmath,mleftright}
\usepackage{xparse}
\usepackage{calc}
\usepackage{braket}
\usepackage{hyperref}
\hypersetup{colorlinks,linkcolor={green},citecolor={magenta},urlcolor={cyan}}
\usepackage{subfig}
\usepackage{amsfonts}
\usepackage{caption}
\usepackage{subcaption}
\usepackage{mathtools}
\usepackage{dsfont}
\usepackage{enumerate}
\usepackage{enumitem}
\usepackage{tikz-cd}
\usepackage{tikz}
\usepackage{tikz-3dplot}
\usetikzlibrary{scopes}
\usetikzlibrary{backgrounds}
\usetikzlibrary{arrows.meta, calc}
\usetikzlibrary{decorations.pathreplacing} 
\usepackage{tabularx}
\usepackage{siunitx}
\usepackage[top=1 in,bottom=1in, left=3cm, right=3cm]{geometry}
\usepackage{placeins}
\usepackage{xcolor}
\usepackage{todonotes}
\usepackage{booktabs}
\usepackage{url}
\usepackage{comment}
\usepackage{caption}
\usepackage{setspace}
\usepackage[bottom]{footmisc}
\usepackage{array}
\usepackage{geometry}
\def\ptrad{2.6pt}   

\newtheorem{proposition}{Proposition}[section]
\newtheorem{lemma}[proposition]{Lemma}
\newtheorem{theorem}[proposition]{Theorem}
\newtheorem{corollary}[proposition]{Corollary}
\newtheorem{conjecture}{Conjecture}[section]
\theoremstyle{definition}
\newtheorem{remark}[proposition]{Remark}
\newtheorem{definition}[proposition]{Definition}
\newtheorem{example}[proposition]{Example}

\newcommand{\R}{\mathbb{R}} 

\newcommand{\Z}{\mathbb{Z}}
\newcommand{\Q}{\mathbb{Q}}

\newcommand{\graphvarphi}{\mathrm{graph}^{\updownarrow}(\varphi)}
\newcommand{\graphphi}{\mathrm{graph}^{\updownarrow}(\phi)}
\newcommand{\graphf}{\mathrm{graph}^{\updownarrow}(f)}

\tikzset{
	axis/.style   = {-{Latex[length=2.5mm]}, thick},
	branch/.style = {thick},
	level/.style  = {blue!80!black, dashed, thick},
	node/.style   = {circle, fill=black, inner sep=1.6pt},
	dot/.style    = {circle, fill=black, inner sep=1.3pt},
}

\title{Algebraically trivial tropical cycles are smash-nilpotent}
\author{Alexia Corradini}
\email{ac2478@cam.ac.uk}

\numberwithin{equation}{section}

\usepackage[backend=biber,style=alphabetic, sorting=nyt]{biblatex}
\begin{document}

\begin{abstract}
	We prove that if $Z$ is a tropical cycle in a tropical variety $Y$ which is algebraically trivial, then for $k\in\Z_{\geq0}$ sufficiently large, $k!\cdot Z^k\subset Y^k$ is rationally trivial. The proof involves explicitly constructing a rational equivalence $k!\cdot(p-q)^k\sim0$ in the product $C^k$, where $p$ and $q$ are any points on $C$. We use this to formulate a tropical version of Voevodsky's conjecture, provide explicit examples, and comment on potential applications to symplectic geometry.
\end{abstract}

\maketitle

\section{Introduction}

The development of tropical geometry over recent years has shown the surprising extent to which its piecewise-linear objects retain shadows of properties from classical algebraic geometry. Our main result in this paper exhibits yet another instance of this phenomenon:

\begin{theorem}\label{main_theorem}
	Let $Y$ be a tropical variety which is regular at infinity, and $Z$ an algebraically trivial $k$-cycle in $Y^{[0]}$, the $0$-sedentarity locus of $Y$ (using terminology from \cite{tropical_book_mikhalkin_rau}). There exits $k_0\in\Z_{>0}$ such that, for all $k\geq k_0$, $$k_0!\cdot Z\times\dots\times Z=:k_0!\cdot Z^k\subset Y^k:=Y\times\dots\times Y$$ is rationally trivial. 
\end{theorem}

This is a tropical version of a Theorem proved independently by Voevodsky in \cite{voevodsky_nilpotence} and Voisin in \cite{voisin_symmetric_cycles}, about algebraic cycles in smooth projective varieties over $\mathbb{C}$. 

\begin{theorem}\cite[Corollary 3.2]{voevodsky_nilpotence}
	Let $Y$ be a smooth projective variety over a field $k$, and $Z\subset Y$ an algebraically trivial cycle of dimension $d$. Then there exits a $k>0$ such that $Z^k=0\in CH_{kd}(Y^k)_{\Q}$.
\end{theorem}
\begin{theorem}\cite[Statement 2.3.1]{voisin_symmetric_cycles}
	Let $Y$ be a variety, and $Z\subset Y\times Y$ a cycle of codimesion $n=\dim(Y)$ which is algebraically equivalent to zero. Then $Z$ is nilpotent in $CH^n(Y\times Y)$ for the product which is composition of correspondences.
\end{theorem}

\begin{remark}
	While Voisin's statement is formulated differently, the key input to its proof -- namely the fact that for any degree zero cycle $z$ on a curve $C$, $z^k$ is rationally trivial in $C^k$ -- implies the statement as phrased by Voevodsky. This is seen by using $\Gamma^{*k}\subset C^k\times (Y\times Y)^k$ instead of $\Gamma^{(k)}\subset C^k\times(Y\times Y)$ in the notation of the proof of \cite[Statement 2.3.1]{voisin_symmetric_cycles}. Then one obtains that $Z^k$ is rationally trivial in $Y^{2k}$ for $k$ sufficiently large, and taking $Z$ to be $Z'\times Z'$ for a cycle $Z'\subset Y$ recovers Voevodsky's formulation. 
\end{remark}

Theorem \ref{main_theorem}, as well as Voevodsky and Voisin's corresponding statements, follow from the special case where $Y=C$ is a smooth compact tropical curve, and $Z=p-q$ the difference of two points in $C$:

\begin{proposition}\label{main_prop}
		Let $C$ be a smooth compact tropical curve of genus $g$. For any two points $p,q\in C$, $k!\cdot(p-q)^{k}$ is rationally trivial for $k$ sufficiently large. 
\end{proposition}

Interestingly, our proof is very different in nature from those of Voevodsky and Voisin, which both rely on the fact that for $k$ sufficiently large compared to $g(C)$, $CH_0(\mathrm{Sym}^k(C))\cong CH_0(\mathrm{Jac}(C))$, and the map $x:CH_0(\mathrm{Sym}^{k}(C))\rightarrow CH_0(\mathrm{Sym}^{k+1}(C))$ given by a point $x\in C$ are isomorphisms. Tropically, fibres of the map $\mathrm{Sym}^k(C)\rightarrow \mathrm{Jac}(C)$ are known to be vastly complicated objects (see \cite{mikhalkin_Zharkov_trop_curves_jac_theta_funct,baker_norine_RR_on_finite_graphs,gathmann_kerber_riemann_roch,baker_faber_metric_abel_jacobi,ABKS2014,numerical_equivalence_tropical,GrossShokriehTothmeresz2022,cools_draisma_payne_robeva_brill_noether,haase_musiker_yu_linear_systems} amongst others), making a direct translation of their argument into tropical geometry difficult.

 Our proof involves explicitly constructing the rational equivalence $k\cdot (p-q)^k\sim0$, without any explicit reference to symmetric powers of $C$ or its Jacobian. A sketch of the proof of Proposition \ref{main_prop}, and a proof of Theorem \ref{main_theorem} using Proposition \ref{main_prop}, is given at the beginning of Section \ref{section:results}. 

\begin{remark}
	The bulk of our constructions in this paper are aimed at proving Proposition \ref{main_prop}. In particular, it does not require the full technical background of abstract tropical cycles, and a reader familiar with tropical curves but not their higher-dimensional generalisations is well-equipped to understand most of our paper. 
\end{remark}

In \cite{voevodsky_nilpotence}, Voevodsky coined the term \textit{smash-nilpotent} to refer to cycles which, after taking sufficiently many self-products, become rationally trivial. For tropical cycles, this translates to:

\begin{definition}\label{def:smash_nilpotence}
	A tropical cycle $Z\subset Y$ is \textit{smash-nilpotent} if there exists $k_0\in\Z_{>0}$ such that, for all $k\geq k_0$, $k!\cdot Z^k\subset Y^k$ is rationally trivial. 
\end{definition}

Theorem \ref{main_theorem} can thus be rephrased as \enquote{algebraically trivial tropical cycles are smash-nilpotent}. Voevodsky's paper \cite{voevodsky_nilpotence} goes a step further by formulating the following conjecture, still the subject of active research today (see for example \cite{kahn_smash_nilpotence_remarks,kahn_sebastian_abelian_3folds,laterveer_smash_nilpotent_examples,sebastian_products_of_curves,sebastian_rationally_connected}):

\begin{conjecture}[Voevodsky's conjecture, \cite{voevodsky_nilpotence}]
	A cycle $Z\subset Y$ is smash-nilpotent if and only if it is numerically trivial. 
	\label{conjecture}
\end{conjecture}

Given our main Theorem \ref{main_theorem}, it seems natural to formulate a tropical version of Voevodsky's conjecture. Tropical numerical equivalence has been defined in various context (see e.g. \cite{numerical_equivalence_tropical}), but to the best of our knowledge never in the presence of a boundary divisor. For this reason, a literal translation of Voevodsky's conjecture requires the assumption that $Y=Y^{[0]}$ has empty boundary, in which case we can formulate it as follows:
\begin{conjecture}[Tropical Voevodsky's conjecture]
	A tropical cycle $Z\subset Y$ is smash-nilpotent if and only if it is numerically trivial.
	\label{tropical_conjecture}
\end{conjecture}
In the presence of a boundary divisor, the weaker conjecture that any \textit{homologically} trivial tropical cycle is smash-nilpotent still makes sense and seems highly non-trivial. 

A hope is for this tropical formulation to be amenable to explicit constructions in a way that could inform efforts to understand Voevodsky's original conjecture. Establishing precise statements about the implication between Voevosky's classical conjecture and its tropical version in various geometric settings would be valuable. Namely, addressing the question of how realizability/tropicalization techniques could prove/disprove Conjecture \ref{conjecture} from a proof/counter-example of Conjecture \ref{tropical_conjecture}. 

An interesting case study for the tropical Conjecture \ref{tropical_conjecture} is the \textit{tropical Ceresa cycle} of a tropical curve $C$: a nullhomologous $1$-cycle $Z(C)$ in the Jacobian $J(C)$ construted by Zharkov in \cite{zharkov_tropical_ceresa}. Zharkov proves that for generic $C$ of genus $3$ whose underlying graph is the complete graph on $4$ vertices ($K4$), the Ceresa cycle $Z(C)$ is \textit{not} algebraically trivial -- following Ceresa's classical result. Classically, Voevodsky's conjecture is known to be true for abelian $3$-folds by work of Kahn--Sebastian in \cite{kahn_sebastian_abelian_3folds}; in particular the Ceresa cycle of \textit{any} genus $3$ curve is smash-nilpotent. Therefore, a natural expectation is that in particular, the tropical Ceresa cycle of any genus $3$ tropical curve should be smash-nilpotent. It was brought to out attention by B. Kahn that their proof provides an explicit bound for the nilpotence level of cycles, which in the case of the Ceresa cycle is $21$. Their paper \cite{kahn_sebastian_abelian_3folds} appeals to the motivic machinery of Kimura finiteness theory, which does not appear to translate immediately to the tropical geometry with its current tools. Nevertheless, a more explicit construction of a rational equivalence $k!\cdot Z(C)^k\sim0$ in $J(C)^k$ seems possible using ideas from our construction. For instance, in Section \ref{section:examples_1} we prove that the Ceresa cycle of any hyper-elliptic curve is smash-nilpotent (this is an application of Theorem \ref{main_theorem}, since Ceresa cycles of hyperelliptic curves are algebraically trivial).

Potential applications to Voevodsky's conjecture were not our original motivation for proving Theorem \ref{main_theorem}. Rather, it was our study of mirror symmetry, which relates algebraic cycles in an algebraic variety to Lagrangian submanifolds in a mirror symplectic manifold, that turned our attention to Voevodsky/Voisin's result. This fits into a more general programme aimed at translating results about algebraic cycles (in this case, Voevodsky/Voisin's Theorem) into new results about Lagrangians, using mirror symmetry. Work of Sheridan--Smith has demonstrated that theorems about Chow groups can be translated into theorems about (cylindrical) Lagrangian cobordism groups in the mirror \cite{nick_ivan1,nick_ivan_K3_rational_equivalence,alvaro_bielliptic}. In our paper \cite{corradini}, we introduced the notion of \textit{algebraic Lagrangian cobordism (or \textit{algcobordism}) groups} as symplectic counterparts to Griffiths groups, with the purpose of translating results about algebraic equivalence of cycles into results about Lagrangians. From this perspective, the naive expectation from Voevodsky/Voisin's Theorem is surprising. It suggests that Lagrangians which are trivial in the alg-cobordism group of a symplectic manifold are \enquote{Lagrangian smash-nilpotent} -- that is, they become trivial in the (cylindrical) Lagrangian cobordism group after taking sufficiently many self-products. Our Theorem \ref{main_theorem} is a natural first step to proving a statement of this nature about Lagrangians. This is because (algebraic) Lagrangian cobordisms are by definition constructed using Lagrangian torus fibrations over tropical abelian varieties, and therefore could be obtained from tropical algebraic/rational equivalences by a \enquote{Lagrangian lift} construction. This will be the subject of future work with D. Koževnikov, in which we prove the following result:
\begin{theorem}[In progress, \cite{danil_and_me}]
	Let $L$ be a Lagrangian submanifold in a symplectic manifold $(X,\omega)$ which is alg-nullcobordant. Then for $k\gg0$, $k!\cdot L^k\subset (X^k,\omega^{\oplus k})$ is cylindrical nullcobordant.
\end{theorem}
Questions of unobstructedness of the Lagrangians in question will be adressed in our work. The starting point for our proof is Theorem \ref{main_theorem}, or more specifically, a Lagrangian lift of the rational equivalence constructed to prove Proposition \ref{main_prop}.

\begin{remark}
	Let us comment on the presence of the $k!$-factor in the statement of Theorem \ref{main_theorem}. It does not appear in Voevosdky/Voisin's original statements, or in Voevodsky's original definition of smash-nilpotence for the following reason. In the classical setting, a Theorem of Roitman states that for a smooth projective variety $X$  over an algebraically closed field of characteristic zero, the Albanese map $$Alb_X:CH_0(X)_{hom}\longrightarrow Alb(X)$$ induces an isomorphism on all torsion subgroups. One readily verifies that $k!\cdot(p-q)^k$ lies in $\ker(Alb_{C^k})$ for points $p,q$ on a curve $C$, which by Roitman's Theorem is torsion-free, hence $(p-q)^k\sim0$. Although we do not address it, the question of whether a tropical version of Roitman's Theorem holds seems worthwhile investigating, and we expect the $k!$-factor could be removed from Theorem \ref{main_theorem} and Definition \ref{def:smash_nilpotence}.
\end{remark}

\begin{remark}
	We also expect the assumption that $Z$ is contained in $Y^{[0]}\subset Y$, i.e. that it does not intersect the boundary divisor, can be removed. In fact, our proof would go through as such. The missing piece comes from the literature on tropical intersection theory, where the proof of Theorem \ref{thm:tropical_voevodsky} in Section \ref{section:results} uses properties of the tropical intersection product which have only been proved away from the divisor. However, for cycles which \textit{are} disjoint from the boundary divisor, our proof implies the stronger fact that the rational equivalence $k_0!\cdot Z^k\sim0$ itself is contained in $(Y^{[0]})^k\subset Y^k$, in a way we will make precise. 
\end{remark}

Section \ref{section:results} contains the proof of Theorem \ref{main_theorem}, starting with an outline of the general strategy which the reader may find helpful for contextualising some of the background introduced in Section \ref{section:background}. Section \ref{section:examples} provides explicit examples and applications of our construction.

\section*{Ackgnowledgements}
A crucial insight behind the proof of Theorem \ref{main_theorem} was inspired by exchanges with Richard Thomas, whom I warmly thank for his engagement and enthusiasm. I am also grateful to Ivan Smith for stimulating discussions about this project, and to Dhruv Ranganathan for helpful comments. I was supported by EPSRC grant EP/X030660/1 for the duration of this work. 

\section*{Notation}

Given a product $X_1\times \dots\times X_n$ of spaces, $\pi_{X_i}$ will always denote the projection onto the $X_i$-factor. 
Throughout this paper, we deal with products of tropical cycles: if $Z$ and $Z'$ are (tropical) cycles which can be written as a sums $Z_1+Z_2$ and $Z_1'+Z_2'$ respectively, then the product $(Z,Z'):=Z\times Z'$ is $$(Z_1+Z_2,Z_1'+Z_2'):=(Z_1,Z_1')+(Z_1,Z_2')+(Z_2,Z_1')+(Z_2,Z_2').$$

\section{Background}\label{section:background}

\subsection{Tropical cycles}

The main objects of tropical intersection theory are (abstract) \textit{tropical cycles}.
These were introduced in \cite{tropical_intersection_theory_allerman_rau}. In this paper, we use the language of \cite[Chapter 7]{tropical_book_mikhalkin_rau}, where tropical cycles live in ambient spaces called \textit{tropical varieties}. We only give a brief overview of the necessary concepts, and refer the reader to \cite{tropical_book_mikhalkin_rau} for a more comprehensive treatment. 

Tropical varieties are defined in \cite[Section 7]{tropical_book_mikhalkin_rau} as a particular kind of \textit{tropical space}. A tropical space is locally modelled on effective balanced open polyhedral sets in the tropical affine space $\mathbb{T}^n:=(\R\cup\{-\infty\})^n$. Following \cite[Definition 7.1.8]{tropical_book_mikhalkin_rau}, a \textit{tropical atlas} on a topological space $X$ is a collection of tuples $(U_i,\psi_i,V_i)_i$ such that the $U_i$ form an open cover of $X$, each $V_i\subset\mathbb{T}^{N_i}$ (for suitable $N_i$) is an effective (i.e. non-negatively weighted), balanced open polyhedral set, each $\psi_i:U_i\rightarrow V_i$ is a homeomorphism, and each transition map $\psi_i\circ\psi_j^{-1}$ is a tropical (i.e. integer affine) isomorphism respecting weights wherever it is defined. A \textit{tropical space} is a topological space together with an equivalence class of such atlases. A \textit{tropical variety} is a tropical space which is Hausdorff, of pure dimension, and of finite type (see \cite[Definition 7.4.14]{tropical_book_mikhalkin_rau} for the definition of finite type). It is said to be \textit{smooth} if it locally looks like a tropical plane in tropical projective space $\mathbb{TP}^n$ (see \cite[Definition 7.4.1]{tropical_book_mikhalkin_rau} for a precise definition). Crucially, smoothness forces all weights to be $1$.

\begin{remark}\label{remark:sedentarity}
	The balancing condition we refer to incorporates the notion of \textit{sedentarity}. Recall that tropicalising a variety requires choosing a boundary divisor, and sedentarity provides a combinatorial way of recording this choice of divisor. For $I\subset\{1,\dots,n\}$, write $R_I:=\{x\in\mathbb{T}^n : x_i=-\infty \text{ for } i\in I\}\cong\R^{n-|I|}$ for the corresponding torus orbit, so that these strata partition $\mathbb{T}^n$; the sedentarity of a point $x\in R_I$ is $\mathrm{sed}(x):=\vert I\vert$, i.e. the number of coordinates of $x$ equal to $-\infty$ (see \cite[Definition 3.1.1]{tropical_book_mikhalkin_rau} for details). 
 	Balancing at a face is then verified by summing over facets of the \textit{same} sedentarity.
	This is one of the main differences with the notion of tropical cycles from \cite{tropical_intersection_theory_allerman_rau}, which are defined using charts to $\R^n$ rather than $\mathbb{T}^n$, meaning all points have sedentarity zero by construction. The technical condition that a tropical variety is \textit{regular at infinity} (see \cite[Definition 7.2.4]{tropical_book_mikhalkin_rau} for a precise definition) guarantees that the $l$-sedentarity loci $X^{[l]}$ are tropical spaces of dimension $\dim(X)-l$ (see \cite[Definition 7.2.9]{tropical_book_mikhalkin_rau}). The main reason we impose this condition in Theorem \ref{main_theorem} is it guarantees that $X^{[0]}$ is an effective tropical cycle in the sense of \cite{tropical_intersection_theory_allerman_rau}. Denoting by $X_\infty:=X^{[1]}=X\backslash X^{[0]}$ the boundary divisor, one can think of regularity at infinity as a \enquote{tropical normal crossings condition}. For instance, a $1$-dimensional tropical space with finitely many $1$-valent vertices (which have sedentarity $1$) is regular at infinity, but this no longer holds if two sedentarity-$1$ vertices at the end of two edges collide to a sedentarity $2$-vertex. 
\end{remark}
A \textit{tropical $k$-cycle} $Z$ in a tropical variety $Y$ (\cite[Definition 6.1.1]{tropical_book_mikhalkin_rau}) is a weighted polyhedral set $Z\subset Y$ (i.e. a finite union of polyhedra) of pure dimension $k$ such that for any chart $U\subset V$, any polyhedral structure on $X\cap U$ and any codimension one cell $\tau\subset X\cap U$, the balancing condition is satisfied. If all its weights are positive, the cycle is said to be \textit{effective}, and can be called a \textit{subvariety} of $Y$. This definition of tropical cycle is designed to agree with the abstract cycle formalism of \cite{tropical_intersection_theory_allerman_rau}, with the important subtlety from Remark \ref{remark:sedentarity} coming from the possibility of points having positive sedentarity in our setting.

A \textit{tropical morphism} $f:X\longrightarrow Y$ of tropical spaces is a continuous map which restricts to a tropical morphism on suitable charts, i.e. it is locally an integer affine map. A tropical cycle $Z$ in $X$ can be pushed forward to a tropical cycle $f_*Z$ in $Y$ (see \cite[Proposition 6.2.1]{tropical_book_mikhalkin_rau} or \cite[Construction 1.6.3]{tropical_intersection_theory_allerman_rau}).

Tropical cycles can be intersected to form new cycles; such a construction was first given in \cite{tropical_intersection_theory_allerman_rau} for cycles in $\R^n$. A generalisation to tropical cycles in any smooth tropical variety was the subject of later work by Allerman in \cite{allermann_intersection_products_smooth_varieties}, although it does not generalise immediately to our notion of tropical cycles from \cite{tropical_book_mikhalkin_rau}, incorporating a boundary divisor. \cite[Section 5.3.1]{tropical_book_mikhalkin_rau} defines the intersection of a cycle of pure sedentarity with a toric boundary divisor which is Cartier (\cite[Definition 5.3.1]{tropical_book_mikhalkin_rau}), which we will appeal to. However, to the best of our knowledge, an intersection $Z\cdot Z'$ of two arbitrary tropical cycles in a tropical variety with boundary divisor has not yet been defined. The special case of cycles in tropical surfaces (i.e. $2$-dimensional tropical varieties) was understood by Shaw in \cite{shaw_tropical_surfaces}. Note that if the ambient variety $Y$ is regular at infinity and the cycles are contained in $Y^{[0]}$, i.e. they do not intersect the boundary divisor (see \cite[Section 7.2]{tropical_book_mikhalkin_rau}), then we are reduced to Allerman's setting, whose intersection product is shown in \cite[Theorem 2.9]{allermann_intersection_products_smooth_varieties} to be commutative, associative, and distributive with respect to addition of cycles amongst other desirable properties. Related constructions of intersection products by Shaw \cite{shaw_matroidal_fans_intersection_product} and subsequently Francois--Rau \cite{francois_rau_diagonal_matroid_varieties} in the setting of matroidal fans/varieties also generalise to smooth varieties.

\subsection{Tropical curves}

A tropical curve is a $1$-dimensional connected tropical variety \cite[Definition 8.1.1]{tropical_book_mikhalkin_rau}. In this paper, we will be interested in \textit{regular} tropical curves, i.e. smooth tropical curves of finite type, because compact such curves are the bases for tropical algebraic equivalences introduced in Section \ref{section:algebraic_equivalence}. 

\begin{remark}
	This notion of regular (compact) tropical curve coincides with that of (compact) tropical curve found in \cite{mikhalkin_Zharkov_trop_curves_jac_theta_funct}. They are allowed one-valent vertices of sedentarity $1$, which is not the case for $1$-dimensional tropical cycles in the sense of  \cite{tropical_intersection_theory_allerman_rau} since such a vertex could never be balanced in $\R^n$ (see Remark \ref{remark:sedentarity}).
\end{remark}

A one-dimensional balanced fan of valence $n$ has a unique smooth local model (up to isomorphism), namely the standard $n$-valent line $L(n)$ whose rays are $-e_1,\dots,-e_{n-1},e_1+\dots+e_{n-1}$, all of weight one. It follows that a smooth tropical curve is entirely determined by its underlying graph together with a choice of integer affine structure on each edge. In dimension one, this coincides with the datum of an inner metric or \textit{edge lengths} $l(E)\in\R\cup\{\infty\}$ for every edge $E$ (notice $l(E)=\infty$ if and only if $E$ is a leaf).

\begin{definition}
	A \textit{metric graph} is a finite connected graph equipped with a complete inner metric on the complement of its $1$-valent vertices (called \textit{ends}). 
\end{definition}

Given such a metric graph, a neighbourhood of each vertex of valence $n$ can be sent via metric-preserving charts to the vertex of $L(n)$, where the length of each primitive generator $e_i$ is set to $1$. Similarly, each finite edge $E$ can be sent via metric-preserving charts to an interval $[0,l(E)]$, and each leaf to $[0;+\infty)$.  One verifies this procedure gives a $1$-to-$1$ correspondence between the set of (isomorphism classes of) (compact) metric graphs and the set of (isomorphism classes of) regular (compact) tropical curves (see \cite[Proposition 8.1.5]{tropical_book_mikhalkin_rau} or \cite[Proposition 3.6]{mikhalkin_Zharkov_trop_curves_jac_theta_funct}).

\subsection{Rational equivalence}\label{section:rational_equivalence}

The definition of rational equivalence adapted to our setting, i.e. one incorporating a boundary divisor, is the one in \cite{tropical_book_mikhalkin_rau}. Much like in the classical setting, it identifies tropical cycles which are fibres of a flat family of cycles over the tropical projective space $\mathbb{TP}^1$, in the following sense:

\begin{definition}\label{def:fibre_over_R}
	Let $Y$ be a tropical variety which is regular at infinity, $W\subset Y\times\mathbb{TP}^1$ a $(k+1)$-subcycle, and for $t\in\R\subset\mathbb{TP}^1$, let $\varphi_t$ be the pullback of the function $\max(x,t)$ along the projection $Y\times\R\rightarrow\R$ (here $x$ is the coordinate on $\R$, while $t$ is a parameter). The \textit{fibre of $W$ at $t$} is the cycle in $Y$ given by $$W_t:=(\pi_Y)_*\mathrm{div}(\varphi_t\vert_{W}),$$
	where $\pi_Y:Y\times\R\rightarrow Y$ is projection onto the first factor. 
	For $t\in\{-\infty,+\infty\}$, \cite[Definition 5.3.4]{tropical_book_mikhalkin_rau} constructs a $k$-cycle by intersecting $W$ with the boundary divisor components $Y\times\{\pm\infty\}$. We denote this by $W_t$ and call this the \textit{fibre of $W$ at $t$}.
\end{definition}

\begin{definition}\cite[Definition 6.3.1]{tropical_book_mikhalkin_rau}
	Let $Y$ be a tropical variety which is regular at infinity, and $Z$, $Z'$ two $k$-cycles in $Y$. We say they are \textit{rationally equivalent} if there exists a $(k+1)$-cycle $W\subset Y\times\mathbb{TP}^1$ such that 
	\begin{enumerate}[label=(\roman*)]
		\item $W$ is the closure of a cycle in $Y\times\R$;
		\item $W_{-\infty}=Z$ and $W_{+\infty}=Z'$.
	\end{enumerate}
	We call the \textit{$k$-th tropical Chow group of $Y$}, and denote by $CH_k(Y)$, the group of tropical $k$-cycles in $Y$ modulo the equivalence relation generated by rational equivalence.
	\label{def:rational_equivalence_MR}
\end{definition}

Importantly, this notion generalises a theory of divisors and rational equivalence which had been well-established in the case of tropical curves \cite{baker_norine_RR_on_finite_graphs,cools_draisma_payne_robeva_brill_noether,gathmann_kerber_riemann_roch,haase_musiker_yu_linear_systems,mikhalkin_Zharkov_trop_curves_jac_theta_funct}. 
Notice that if the cycles $Z$ and $Z'$ are contained in the complement of the boundary divisor $Y^{[0]}\subset Y$, then the condition that $W$ is the closure of a cycle in $Y\times\R$ implies that there must exist $a<b\in\R$ such that, for any $t<a$ and $t'>b$, $W_{t}=Z$ and $W_{t'}=b$. In fact, under this condition, $Z$ and $Z'$ are abstract tropical cycles in $Y^{[0]}$ in the sense of \cite{tropical_intersection_theory_allerman_rau}, and this notion of rational equivalence coincides with the one previously introduced in \cite{tropical_rational_equivalence_allerman_hampe_rau}:

\begin{definition}[{\cite[Definition 3.4]{tropical_intersection_theory_allerman_rau}}]
	Two $k$-cycles $Z$ and $Z'$ in $Y^{[0]}$ are \textit{rationally equivalent over $\R$} if and only if there exists a $(k+1)$-cycle $W\subset X\times\R$ and points $a<b \in \R$ with $W_t-W_{t'}=Z-Z'$ for all $t\leq a$ and $t'\geq b$. 
	\label{def:rational_equivalence_over_R}
\end{definition}

An equivalent definition can be given in terms of divisors of rational functions, which we will use in \ref{section:results}. 

\begin{definition}[{\cite[Definition 6.1]{tropical_intersection_theory_allerman_rau}}]
	A rational function $\phi$ on a tropical cycle $X$ is a continuous function $\phi:X\longrightarrow\R$ that is integer affine on each cell of a suitable polyhedral structure of $X$. 
\end{definition}

To any such $f$, one can associate a divisor $\mathrm{div}(f)$, which is a weighted subcomplex of $X$ of codimension one. The general idea is that the support of the divisor records the locus on which $f$ is \textit{not} affine, with weights quantifying the change in slope. A rigorous definition requires introducing the \textit{balanced graph} of $f$. Because this is a construction we will slightly modify for our purposes, we defer the details to Section \ref{section:balanced_graph}. Divisors of rational functions are the starting point to formulate an equivalent definition to Definition \ref{def:rational_equivalence_over_R}.

\begin{definition}[{\cite[Definition 1.7.1]{allerman_thesis}}]
	A tropical cycle $Z\subset Y^{[0]}$ is said to be \textit{bounded rationally equivalent to zero} if there exists a morphism $f:X\rightarrow Y$ and a bounded rational function $\phi:X\rightarrow\R$ such that $f_*\mathrm{div}(\phi)=Z$ (i.e. $\dim(X)=\dim(Z)+1)$). 
	We say that $Z$ and $Z'$ are \textit{bounded rationally equivalent} if $Z-Z'$ is bounded rationally equivalent to zero. 
	\label{def:rational_equivalence}
\end{definition}

The Proposition below is proved in \cite{tropical_rational_equivalence_allerman_hampe_rau} for cycles in $\R^n$. However, their proof readily extends to abstract tropical cycles in the sense of \cite{tropical_intersection_theory_allerman_rau}. We sketch it in Section \ref{section:balanced_graph}, after introducing the relevant notion of \textit{two-ended balanced graph} of a rational function. 

\begin{proposition}[{\cite[Proposition 3.5]{tropical_rational_equivalence_allerman_hampe_rau}}]\label{prop:equivalence_rational_equivalence_definitions}
	Two $k$-cycles $Z$ and $Z'$ in $Y^{[0]}$ are bounded rationally equivalent if and only if they are rationally equivalent over $\R$.  
\end{proposition}

\begin{remark}
	It is clear from Definition \ref{def:rational_equivalence} that the pushforward $g_*Z\subset Y'^{[0]}$ of a rationally trivial cycle $Z\subset Y^{[0]}$ by a morphism $g:Y\rightarrow Y'$ is rationally trivial by considering the morphism $g\circ f:X\rightarrow Y'$. 
	\label{rmk:pushforward_on_Chow}
\end{remark}

\subsection{The two-ended balanced graph construction}\label{section:balanced_graph}

The notion of \textit{balanced graph} of a rational function $\varphi: C \rightarrow \mathbb{R}$ on a tropical cycle was first introduced by Allermann--Rau in \cite[Section~3]{tropical_intersection_theory_allerman_rau}, to construct the Weil divisor associated to a Cartier divisor. Their construction, as well as our modification of it, are special instances of the more general notion of \textit{tropical modifications} (see e.g. \cite[Section 4.5]{tropical_book_mikhalkin_rau}); these are constructed for arbitrary tropical cycles, while we consider only tropical curves. The starting observation is that the set-theoretic graph of $\varphi$ inside $C \times \mathbb{R}$ is not tropical: the balancing condition fails on the codimension one locus on which $\varphi$ is not affine, as Figure~\ref{fig:not_balanced} illustrates.

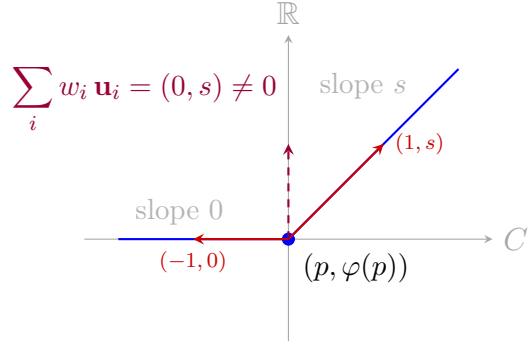
\begin{figure}[h!]
	\centering
	\begin{tikzpicture}[scale=0.9, >=stealth]
		\draw[->, gray!70] (-3,0) -- (3,0) node[right] {$C$};
		\draw[->, gray!70] (0,-1.5) -- (0,3) node[above] {$\mathbb{R}$};
		\node[above, gray!60] at (-1.6, 0.05) {\small slope $0$};
		\node[above left, gray!60] at (1.9, 1.9) {\small slope $s$};
		\draw[thick, blue] (-2.5, 0) -- (0, 0) -- (2.5, 2.5);
		\filldraw[blue] (0,0) circle (2.5pt);
		\node[below right] at (0.05,-0.05) {\small $(p,\varphi(p))$};
		\draw[->, thick, red!80!black] (0,0) -- (-1.4, 0)
		node[below] {$\scriptstyle (-1,\,0)$};
		\draw[->, thick, red!80!black] (0,0) -- (1.4, 1.4)
		node[right] {$\scriptstyle (1,\,s)$};
		\draw[->, dashed, thick, purple!80!black] (0,0) -- (0, 1.4)
		node[above left] {\small $\displaystyle\sum_i w_i\,\mathbf{u}_i = (0,s) \neq 0$};
	\end{tikzpicture}
	\caption{The set-theoretic graph $\operatorname{graph}(\varphi)\subset C\times\mathbb{R}$.
		At the vertex $(p,\varphi(p))$, the weighted sum of primitive outgoing lattice vectors
		$(-1,0)$ and $(1,s)$ (with $s\neq 0$ the slope of $\varphi$ to the right of $p$) equals
		$(0,s)\neq 0$, so the balancing condition fails.}
	\label{fig:not_balanced}
\end{figure}

To remedy this, they add semi-infinite edges with appropriate weights. More precisely, at each vertex $v$ of $C$ where $\varphi$ fails to be locally affine, the defect in the balancing condition at the lifted vertex $(v, \varphi(v)) \in C \times \mathbb{R}$ is measured by
\[
\delta(v) \coloneqq \sum_{e \ni v} w(e) \cdot s_{e,v},
\]
where the sum ranges over all edges $e$ of $C$ adjacent to $v$, $w(e)$ denotes the weight of $e$, and $s_{e,v}$ is the slope of $\varphi$ along $e$ in the direction pointing away from $v$. Note that $\delta(v)$ equals the coefficient of $[v]$ in the divisor $\mathrm{div}(\varphi)$. One then adjoins to $(v, \varphi(v))$ a semi-infinite edge in the direction $(0, \pm1)$ carrying weight $\vert\delta(v)\vert$, where the sign $\pm$ is $-1$ if $\delta(v)>0$ and $-1$ if $\delta(v)<0$. The \textit{two-ended balanced graph of $\varphi$}, which we denote $\graphvarphi$, is the union of $\operatorname{graph}(\varphi)$ -- in which edges inherit weights from the corresponding edges of $C$, with these weighted semi-infinite edges. Note that $\graphvarphi$ is effective, i.e. all weights are positive. This is the reason for our modification of Allerman--Rau's original construction; their balanced graph is obtained by adjoining only \textit{negative} semi-infinite edges (primitive vectors $(0,-1)$) with weights $\delta(v)$, i.e. possibly negative. These two constructions are represented in Figure~\ref{fig:two_ended} below, in the simple case where $C \sim S^1$.

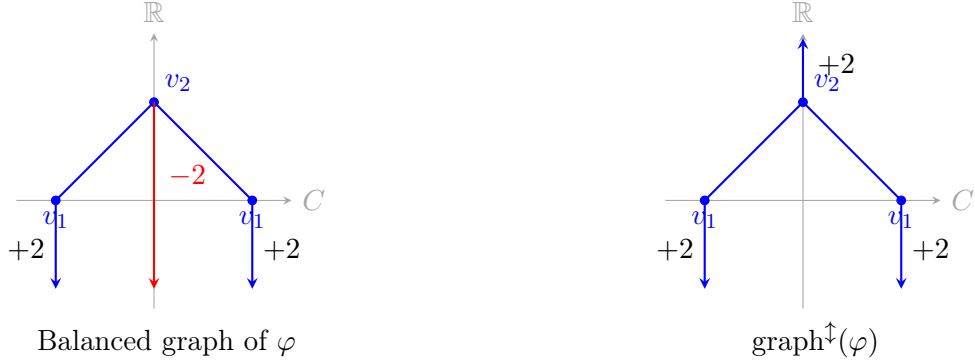
\begin{figure}[h!!]
	\centering
	\begin{minipage}[b]{0.48\textwidth}
		\centering
		\begin{tikzpicture}[scale=0.65, >=stealth]
			\draw[->, gray!70] (-2.8,0) -- (2.8,0) node[right] {\small $C$};
			\draw[->, gray!70] (0,-2.2) -- (0,3.4) node[above] {\small $\mathbb{R}$};
			\draw[thick, blue] (-2,0) -- (0,2) -- (2,0);
			\filldraw[blue] (-2,0) circle (2.5pt) node[below] {\small $v_1$};
			\filldraw[blue] ( 0,2) circle (2.5pt) node[above right] {\small $v_2$};
			\filldraw[blue] ( 2,0) circle (2.5pt) node[below] {\small $v_1$};
			\draw[thick,blue, ->] (-2, 0) -- (-2, -1.8);
			\node[left] at (-2, -1.0) {\small $+2$};
			\draw[thick, blue, ->] ( 2, 0) -- ( 2, -1.8);
			\node[right] at ( 2, -1.0) {\small $+2$};
			\draw[thick, red, ->] (0, 2) -- (0, -1.8);
			\node[right, red] at (0.08, 0.5) {\small $-2$};
		\end{tikzpicture}
		\par\smallskip
		\small Balanced graph of $\varphi$
	\end{minipage}
	\hfill
	\begin{minipage}[b]{0.48\textwidth}
		\centering
		\begin{tikzpicture}[scale=0.65, >=stealth]
			\draw[->, gray!70] (-2.8,0) -- (2.8,0) node[right] {\small $C$};
			\draw[->, gray!70] (0,-2.2) -- (0,3.4) node[above] {\small $\mathbb{R}$};
			\draw[thick, blue] (-2,0) -- (0,2) -- (2,0);
			\filldraw[blue] (-2,0) circle (2.5pt) node[below] {\small $v_1$};
			\filldraw[blue] ( 0,2) circle (2.5pt) node[above right] {\small $v_2$};
			\filldraw[blue] ( 2,0) circle (2.5pt) node[below] {\small $v_1$};
			\draw[thick, blue, ->] (-2, 0) -- (-2, -1.8);
			\node[left] at (-2, -1.0) {\small $+2$};
			\draw[thick, blue, ->] ( 2, 0) -- ( 2, -1.8);
			\node[right] at ( 2, -1.0) {\small $+2$};
			\draw[thick, blue, ->] (0, 2) -- (0, 3.3);
			\node[right] at (0.08, 2.75) {\small $+2$};
		\end{tikzpicture}
		\par\smallskip
		\small $\graphvarphi$
	\end{minipage}
	\caption{A rational function $\varphi$ on $C \sim S^1$ with two vertices $v_1, v_2$ and
		$\operatorname{div}(\varphi) = 2[v_1] - 2[v_2]$. 
		In the balanced graph (left), the semi-infinite edge at $v_2$
		carries negative weight $-2$ (shown in red). In the two-ended balanced graph $\graphvarphi$
		(right), this is replaced by an upward ray with positive weight $+2$.}
	\label{fig:two_ended}
\end{figure}

\begin{proof}[{Proof of Proposition \ref{prop:equivalence_rational_equivalence_definitions}}]
	This is a sketch of an immediate adaptation of the proof of Proposition 3.5 in \cite{tropical_rational_equivalence_allerman_hampe_rau}. Assume $Z$ and $Z'$ are bounded rationally equivalent, namely $Z-Z'=f_*(\mathrm{div}(\phi))$, where $f:X\rightarrow Y$ is a tropical map from $X$ with $\dim(X)=k+1$ and $\phi$ a bounded rational function on $X$. The \textit{two-ended balanced graph} of $\varphi$, $\graphphi$, introduced in Section \ref{section:balanced_graph} is a $(k+1)$-dimensional tropical cycle in $X\times\R$. Let $W:=(f\times\text{id})_*\graphphi\subset X\times\R$. 
	Because $\phi$ is bounded, one can take $a<\min(\phi)$ and $b>\max(\phi)$, and by construction of $\graphphi$, $W_t-W_{t'}=Z-Z'$ for any $t<a$ and $t'>b$. 
	On the other hand, assume the existence of a $(k+1)$-cycle $W\subset X\times\R$ satisfying the conditions above. Consider the bounded rational function $\varphi_{[a,b]}:=\max\{a,x\}-\max\{b,x\}$ on $\R$, and its pullback $\pi_\R^*\varphi:W\longrightarrow\R$ to $W$ via the projection map $\pi_\R:Y\times\R\longrightarrow\R$. Now consider the first projection $\pi_Y:Y\times\R\longrightarrow Y$. Restricting it to $W$ yields a morphism $\pi_Y:W\longrightarrow Y$. Unpacking the definitions yields $$(\pi_Y)_*\mathrm{div}(\pi_\R^*\varphi_{[a,b]})=Z-Z'.$$ 
\end{proof}

\subsection{Algebraic equivalence}\label{section:algebraic_equivalence}

We use the definition of algebraic equivalence introduced in \cite[Section 6.4]{tropical_book_mikhalkin_rau}, which also appears in \cite{zharkov_tropical_ceresa}. 

Classically, algebraic equivalence is defined by identifying fibres of flat families of cycles over smooth projective curves -- although equivalent definitions replace the base curve by an abelian variety, or any smooth (projective) variety. 
In the tropical setting, we want to consider families of tropical cycles over smooth compact tropical curves -- i.e. smooth compact tropical varieties of dimension one. For this, we must extend Definition \ref{def:fibre_over_R} to a notion of \enquote{fibre over a point $p\in C$}, where $C$ is such a curve.

Let $Y$ be a smooth tropical variety, $C$ a smooth tropical curve, and $Z\subset Y\times C$ a $(k+1)$-cycle. Assume $p\in C^{[0]}$ has sedentarity zero. Then a neighbourhood $U$ of $p$ looks like the line $L\subset \R^n$ with vertex sitting at zero, and denoting by $\varphi_p$ the pull back of $\max(x_1,\dots,x_n,0)$ to $Y\times U$ we take the $k$-cycle $Z_p=\mathrm{div}(\varphi_p)$. If $(\pi_Y)_*Z$ is furthermore contained in $Y^{[0]}$, this coincides with Allerman's notion of intersection of cycles $Z_p:=C\cdot (V\times\{p\})$ \cite{allermann_intersection_products_smooth_varieties}. If $p$ has sedentarity $1$, then $C$ looks like $\mathbb{T}$ locally around $p$, with $p$ mapped to $\{-\infty\}$. Assuming $Z$ is the closure of a cycle over $Y\times(C\backslash\{p\})$, \cite[Definition 5.3.4]{tropical_book_mikhalkin_rau} constructs an intersection with the \enquote{local divisor at infinity} $Y\times\{p\}$. This gives a $k$-cycle which we denote $Z_p$.

\begin{definition}[Definition 6.4.1 in \cite{tropical_book_mikhalkin_rau}]
	An \textit{algebraic equivalence} between two cycles $Z_1$ and $Z_2$ in a tropical variety $Y$ is given by a smooth compact tropical curve $C$ with two points $p,q\in C$ and a $(k+1)$-cycle $W\subset Y\times C$ such that
	\begin{enumerate}[label=(\roman*)]
		\item $W$ is the closure of a cycle in $V\times C\backslash\{p,q\}$;
		\item $W_p=Z_1$ and $W_q=Z_2$.
	\end{enumerate}
	We call \textit{algebraic equivalence} the equivalence relation generated by algebraic equivalences, and write $Z_1\stackrel{alg.}{\sim}Z_2$ if $Z_1$ and $Z_2$ are algebraically equivalent.
\end{definition}

\section{Results}\label{section:results}

In this Section we prove our main result:

\begin{theorem}\label{thm:tropical_voevodsky}
	Let $Y$ be a smooth tropical variety which is regular at infinity, and $Z\subset Y^{[0]}$ a cycle which is algebraically trivial in $Y^{[0]}$. There exists a positive integer $k_0$ such that, for every integer $k\geq k_0$, $$k!\cdot Z\times\dots\times Z=:k!\cdot Z^k\subset Y^k:=Y\times\dots\times Y$$ is rationally trivial. 
\end{theorem}

\subsection{Outline of proof}\label{section:outline_of_proof}

Theorem \ref{main_theorem} follows from the case where $Y=C$ is a smooth compact tropical curve, and $Z\subset C$ is the nullhomologous $0$-cycle $p-q$ for two points $p,q\in C$:

\begin{proposition}\label{prop:tropical_voevodsky_curves}
	Let $C$ be a smooth compact tropical curve of genus $g$. For any two points $p,q\in C^{[0]}$, $k!\cdot(p-q)^{k}$ is rationally trivial in $(C^{[0]})^k$ for $k$ sufficiently large. 
\end{proposition}

We realise this rational equivalence $k!\cdot(p-q)^{k}$ in $(C^{[0]})^k$ by constructing a $1$-cycle $W\subset C^k\times\R$ as in Definition \ref{def:rational_equivalence_over_R}.

Assuming Proposition \ref{prop:tropical_voevodsky_curves}, one can conclude as follows. 

\begin{proof}[Proof of Theorem \ref{thm:tropical_voevodsky} assuming Proposition \ref{prop:tropical_voevodsky_curves}.]
	
	Let $Z$ be an algebraically trivial tropical $l$-cycle in $Y^{[0]}$. By definition, there exists a smooth compact tropical curve $C$ with two points $p$ and $q$, and an $(l+1)$-cycle $\Gamma\subset Y^{[0]}\times C$ such that $\Gamma_p=Z$ and $\Gamma_q=0$. We argue that $p$ and $q$ can be taken to have sedentarity zero. To see this, assume $\mathrm{sed}(p)=1$, i.e. $p$ is a $1$-valent vertex of $C$. Then $p$ can be identified with $\{-\infty\}$ in a neighbourhood of the form $[-\infty;a]$ for some $a\in\R$. Recall that $\Gamma\subset Y^{[0]}\times C$ is the closure of a cycle in $Y^{[0]}\times(C\backslash\{p\})$. This means that, unless $\Gamma$ is of the form $Z\times(-\infty;a)$ for suitable $a\in\R$ near $p$, $Z$ would have to intersect the boundary divisor $Y_\infty\subset Y$, contradicting our assumption that $Z\subset Y^{[0]}$. This means any $p\in(-\infty;a)$ satisfies $\mathrm{sed}(p)=0$ as well as $\Gamma_p=Z$. Following this observation, we will also simply denote $C$ to imply $C^{[0]}$ in the rest of the proof, in order to simplify notation. 
	
	Because all intersections products occuring below take place within $0$-sedentarity locus, they are defined following \cite{allermann_intersection_products_smooth_varieties}. Furthermore, a rational equivalence in $Y^{[0]}$ yields a rational equivalence in $Y$ by pushing-forward by the inclusion. For this reason, from now on we assume $Y=Y^{[0]}$ to simplify notation.

	By Proposition \ref{prop:tropical_voevodsky_curves}, for $k$ sufficiently large, there exists a rational equivalence $k!\cdot(p-q)^k\sim 0$ in $C^k$, which by Definition \ref{def:rational_equivalence_over_R} can be realised by a $1$-cycle $W\subset C^k\times\R$ satisfying $W_t-W_{t'}=k!\cdot(p-q)^k$ for $t\ll0\in\R$ and $t'\gg0\in\R$.
	
	 	We will now consider the product $\Gamma^k:=\Gamma\times\dots\times\Gamma$, for $k\in\Z_{\geq1}$. The idea is that using $\Gamma^k$ -- viewed as a cycle in $Y^k\times C^k$ after permuting the factors accordingly -- to \enquote{push-forward/compose} (in the sense of correspondences) the rational equivalence $W\subset C^k\times\R$ should yield a cycle $\Gamma^k(W)\subset Y^k\times\R$ witnessing a rational equivalence $$k!\cdot(\Gamma_p-\Gamma_q)^k=k!\cdot Z^k\sim 0. $$

	Let us make this precise, denoting by $\Delta_\R\subset\R\times\R$ the diagonal. Consider the tropical cycle $(\Gamma^k\times\Delta_\R)$ inside $Y^k\times C^k\times\R^2$, which we view as a cycle in $Y^k\times\R\times C^k\times\R$ after permuting the factors accordingly. 
	Then $$\Gamma^k(W):=(\pi_{Y^k\times\R})_*\left((\Gamma^k\times\Delta_\R)\cdot(Y^k\times\R\times W)\right)\subset Y^k\times\R$$ is a well-defined $(kl+1)$-cycle, as the intersection takes place within the $0$-sedentarity locus.

	We now show that, for any $a\in\R$, the fibres of $\Gamma^k(W)$ compute the composition in the expected sense, i.e.
	\begin{equation}\label{fibre_of_composition}
		\Gamma^k(W)_a=(\pi_{Y^k})_*\big(\Gamma^k\cdot(Y^k\times W_a)\big)=:\Gamma^k(W_a).
	\end{equation}
	By Definition \ref{def:fibre_over_R}, and unwinding the definition of $\Gamma^k(W)$ above,
	\begin{align}
		\Gamma^k(W)_a&=(\pi_{Y^k})_*\Big(\mathrm{div}(\max(x,a))\cdot\Gamma^k(W)\Big)\notag\\
		&=(\pi_{Y^k})_*\Big(\mathrm{div}(\max(x,a))\cdot(\pi_{Y^k\times\R})_*\big[(\Gamma^k\times\Delta_\R)\cdot(Y^k\times\R\times W)\big]\Big), \label{eq:unwind}
	\end{align}
	where $\max(x,a)$ is pulled back from $\R$ to $Y^k\times\R$. 
	
	Recall the projection formula for the intersection product on smooth tropical varieties \cite{tropical_intersection_theory_allerman_rau,allermann_intersection_products_smooth_varieties}: if $f:X\to X'$ is proper and $D$ is a Cartier divisor on $X'$, then for any cycle $\alpha$ on $X$,
	\begin{align}
		f_*\big(f^*D\cdot\alpha\big)=D\cdot f_*\alpha.\notag
	\end{align}
	After pulling back $\max(x,a)$ further to $Y^k\times\R\times C^k\times\R$ via $\pi_{Y^k\times\R}$, this allows us to move $\mathrm{div}(\max(x,a))$ inside the push-forward in Equation \eqref{eq:unwind}:
	\begin{align}
		\Gamma^k(W)_a=\big(\pi_{Y^k}\circ\pi_{Y^k\times\R}\big)_*\Big(\mathrm{div}(\max(x,a))\cdot(\Gamma^k\times\Delta_\R)\cdot(Y^k\times\R\times W)\Big). \notag
	\end{align}
	
	 By commutativity and associativity of the intersection product \cite{tropical_intersection_theory_allerman_rau,allermann_intersection_products_smooth_varieties},
	\begin{align}
		\mathrm{div}(\pi_{Y^k\times\R}^*\max(x,a))\cdot(\Gamma^k\times\Delta_\R)\cdot(Y^k\times\R\times W)&=(\Gamma^k\times\Delta_\R)\cdot\left(\mathrm{div}(\max(x,a)) \cdot (Y^k\times\R\times W)\right) .\notag
	\end{align}
	
	Notice that $\max(x,a)$ is viewed here as a function on $Y^k\times\R$ having been pulled back from $\R$. Write $q_1,q_2:\R\times\R\to\R$ for the two projections, and $\Delta_\R=\{(t,t):t\in\R\}$. Because $q_1=q_2$ on $\Delta_\R$ (both restrict to the identity under $\Delta_\R\cong\R$), the divisors $\mathrm{div}(q_1^*\max(x,a))$ and $\mathrm{div}(q_2^*\max(x,a))$ agree once intersected with $\Delta_\R$. This means that the presence of the diagonal in the previous equation allows us to equally view $\mathrm{div}(\max(x,a))$ as the divisor of $\max(x,a)$ pulled back via the \textit{second} $\R$-factor, appearing in $C^k\times\R$. This yields $$\mathrm{div}(\pi_{ C^k\times\R}^*\max(x,a))=Y^k\times\R\times\mathrm{div}(\max(x,a)),$$ hence
	\begin{align}
		\mathrm{div}(\pi_{C^k\times\R}^*\max(x,a)) \cdot (Y^k\times\R\times W)&=Y^k\times\R\times(\mathrm{div}(\max(x,a))\cdot W) \notag
		&=Y^k\times\R\times W_a\times\{a\}. \notag
	\end{align}

	Again using that, along both factors, the cycles $\R$ and $\{a\}$ coincide after intersection with $\mathrm{div}(\max(x,a))$ and $\Delta_\R$, we find
	\begin{align}
		\Gamma^k(W)_a=&(\pi_{Y^k})_*\left((\Gamma^k\times\{a,a\})\cdot(Y^k\times\{a\}\times W_a\times\{a\})\right) \notag\\
		=&(\pi_{Y^k})_*\left(\Gamma^k\cdot(Y^k\times W_a)\right)=: \Gamma^k(W_a). \notag
	\end{align}

	Applying this at $t\ll0$ and $t'\gg0$ as above, i.e. such that $W_t=k!\cdot(p-q)^k$, $W_{t'}=0$, we get
	\begin{equation}\label{difference_of_fibres}
		\Gamma^k(W)_t-\Gamma^k(W)_{t'}=\Gamma^k(W_t)-\Gamma^k(W_{t'})=\Gamma^k\big(k!\cdot(p-q)^k\big)-\Gamma^k(0)=\Gamma^k\big(k!\cdot(p-q)^k\big). \notag
	\end{equation}
	By Definition \ref{def:rational_equivalence_over_R}, this exhibits $\Gamma^k(W)\subset Y^k\times\R$ as a rational equivalence between $\Gamma^k\big((p-q)^k\big)$ and $0$ in $Y^k$.

	It remains to identify $\Gamma^k\big((p-q)^k\big)$. 
	For any point $(p_1,\dots,p_k)\in C^k$, the fibre $\Gamma^k_{(p_1,\dots,p_k)}$ of $\Gamma^k$ over $(p_1,\dots,p_k)$ is simply the product $\Gamma_{p_1}\times\dots\times\Gamma_{p_k}$ of the fibres of $\Gamma$ over the $p_i$'s. 
	
	Expanding $(p-q)^k$ multilinearly as a signed sum of terms of the form $$\sigma\big(\underbrace{p,\dots,p}_{i},\underbrace{q,\dots,q}_{k-i}\big),$$ where $0\leq i\leq k$ and $\sigma\in S_k$, linearity of $\Gamma^k(-)$ means $\Gamma^k\big((p-q)^k\big)$ can be expanded as a signed sum of terms of the form $$\Gamma^k_{\sigma(p,\dots,p,q,\dots,q)}=\Gamma_p^{\times i}\times\Gamma_q^{\times(k-i)},$$
	up to reordering the factors according to $\sigma$. 
	Because $\Gamma_q=0$, every summand with $i<k$ vanishes; only the term $i=k$, with coefficient $(-1)^0=1$ and a single permutation $\sigma=\mathrm{id}$, survives, giving $\Gamma^k\big((p-q)^k\big)=\Gamma_p^{\times k}=Z^k$.

	Combining this with the previous equation, $\Gamma^k(W)$ realises a rational equivalence between $Z^k$ and $0$ in $Y^k$, i.e.\ $Z^k\sim0$.
	
\end{proof}

It now remains to prove Proposition \ref{prop:tropical_voevodsky_curves}: the strategy for doing so is the following. Let $C$ be a smooth compact tropical curve of genus $g$, and $p,q\in C^{[0]}$. For $k>g$, the divisor $k\cdot p-q$ has degree $k-1\geq g$, therefore by \cite[Corollary 6.6]{mikhalkin_Zharkov_trop_curves_jac_theta_funct} its linear equivalence class contains an effective divisor $D$. More precisely, there exists a rational function such that $k\cdot p-q-D=\mathrm{div}(\varphi)$. Note the condition that the rational equivalence is in $C^{[0]}\subset C$ restricts the support of $D$ to $C^{[0]}\subset C$; while this does not follow immediately from \cite[Corollary 6.6]{mikhalkin_Zharkov_trop_curves_jac_theta_funct}, it follows from our construction below. 

From this rational equivalence, and more precisely from the level sets of $\varphi$, we construct a one-parameter family of $k$ points of $C$, $A(t)=\{x_1(t),\dots,x_k(t)\}$, which satisfies the following:
\begin{enumerate}[label=(\roman*)]
	\item there is a finite set $\{a=t_0<\dots<t_r=b\}$ of real numbers such that $A(t)$ is well-defined up to a choice of permutation for each connected component in $\R\backslash\{t_0,\dots,t_r\}$;
	\item the permutations associated with each of these connected components can be chosen so that every $x_i(t)\in C$ is continuous;
	\item for any $t<a$, $x_1(t)=\dots=x_k(t)=p$;
	\item for any $t>b$, $x_i(t)=q$ for some $i$.
\end{enumerate}
The family $A(t)$ is constructed in Section \ref{section:family_k_points}. Roughly, it is obtained by taking the $C$-coordinates of points in $$\graphvarphi_t:=\graphvarphi\cap (C\times\{t\})$$ with a suitable notion of multiplicity. 

From the $1$-parameter family $A(t)$, we will construct a tropical $1$-cycle $W\subset C^k\times\R$, which is defined precisely so that its fibres are given by
\begin{align}
	W_t:=\bigoplus_{\sigma\in S_k} (x_{\sigma(1)}(t)-q,\dots,x_{\sigma(k)}(t)-q)\subset C^k \notag
\end{align}
for any $t\in\R$ -- its precise construction is the subject Section \ref{section:constructing_W}. Given $a,b\in\R$ appearing in properties $(iii), (iv)$ of $A(t)$ above, $W_t=k!\cdot(p-q)^k$ for any $t>a$ and $x_i(t)-q=0$ hence $W_{t}=0$ for any $t>b$. In this way, $W$ gives a rational equivalence $k!\cdot(p-q)^k\sim0$ in $C$. From our construction, all points in $A(t)$ will lie in $C^{[0]}\subset C$, allowing for $W\subset (C^{[0]})^k\times\R$. Section \ref{section:constructing_W} is dedicated to constructing $W$ and understanding its structure as a tropical $1$-cycle.

\subsection{Constructing the family of points $A(t)_{t\in\R}$}\label{section:family_k_points}

The starting point for constructing $A(t)$ -- a $1$-parameter family of $k$ points of $C^{[0]}$ -- is the rational equivalence $k\cdot p-q\sim \mathrm{div}(\varphi)$. More precisely, we use the two-ended graph of $\varphi$, $\graphvarphi$, constructed in Section \ref{section:balanced_graph}. Points in $A(t)$ can be thought of as the $C$-coordinates of points in $\graphvarphi_t:=\graphvarphi\cap(C\times\{t\})$ with a certain \textit{multiplicity}; we make this precise here.

 Given any $t\in\R$, there are two options for the intersection of $C\times\{t\}$ with $\graphvarphi$:
 \begin{enumerate}
 	\item it is transverse, therefore contains a finite number of points;
 	\item $t=\varphi(e)$ for some segment $e$ of $C$ on which $\varphi$ is constant. 
 \end{enumerate}
 
For our construction to work, we will need to avoid the latter case, with the exception of the leaves of $C$. Indeed, recall that the leaves of $C$ have infinite length, so it cannot take finite values at the $1$-valent vertices of $C$ unless it is constant at the end of its leaves. We use the following Proposition:

\begin{proposition}\label{non-zero_slopes}
	Let $p$ and $q$ be points on a smooth compact tropical curve $C$. For sufficiently large positive integer $k$, the divisor $k\cdot p-q$ is rationally equivalent to an effective divisor $D$, and the rational equivalence $$k\cdot p-q-D=\mathrm{div}(\varphi)$$ can be realised by a rational function $\varphi$ whose slopes are non-zero, except for connected segments adjacent to the $1$-valent vertices of $C$ (where we force $\varphi$ to be constant).
\end{proposition}

\begin{remark}
	Forcing $\varphi$ to be constant at the end of the leaves of $C$ will make the construction of the rational equivalence $W\subset C^k\times\R$ between $(p-q)^k$ and $0$ slightly more cumbersome. One could try to view $\varphi$ as a function with values in $\R\cup\{\pm\infty\}$ instead. The reason we could not use this approach, is that tropical intersection theory is not well understood for tropical cycles which intersect the boundary divisor. Although a $1$-valent vertex $x$ has pure sedentarity, hence we could still make sense of the \textit{fibre over $x$} in an algebraic equivalence in $Y\times C$ by using \cite[Section 5.3.1]{tropical_book_mikhalkin_rau}, the proof of Theorem \ref{main_theorem} uses properties of the intersection product which are only known within the zero-sedentarity locus. 
\end{remark}

The proof Proposition \ref{non-zero_slopes} relies on the following Lemma:

\begin{lemma}\label{voltage_functions}
	For any point $x$ of a smooth compact tropical xcurve $C$, there exists a choice of point $y_i$ for every edge $E_i$ of $C$, a rational function $f$ whose divisor is $\mathrm{div}(f)=(k_1+\dots+k_s)\cdot x-k_1\cdot y_1-\dots-k_s\cdot y_s$ for some integers $k_1,\dots,k_s$, and whose slopes are non-zero exactly outside of connected segments adjacent to each $1$-valent vertex of $C$. 
\end{lemma}
\begin{proof}
	We will start by constructing a PL function with \textit{real} slopes, then show that by choosing the points $y_i\in E_i$ appropriately these slopes can be taken to be both non-zero and rational, from which point we can multiply this function by a sufficiently large integer to obtain the desired result. 
	
	We will abuse notation and write $\mathrm{div}(f)$ for the obvious extension of the definition of divisor of any PL function $f$, where the slopes are allowed to be any real number (and therefore, so are the coefficients of $\mathrm{div}(f)$).
	
	Certain elements of our construction are similar to that of \cite[Corollary 3]{baker_faber_metrized_graphs}, who show the existence of a piecewise-linear function (with \textit{real} slopes) $f_{x,y}$ whose divisor $\mathrm{div}(f_{x,y})$ is $x-y$ for any pair of points $x$ and $y$ in $C$. 
	
	Fix a vertex set $V(C)=\{v_1,\dots,v_{\vert V(C)\vert}\}$ for $C$ with $v_1=x$, in the sense of \cite{baker_faber_metrized_graphs}. That is, $C\backslash V(C)$ is a finite disjoint union of subspaces isometric to open intervals, whose topological closure in $C$ is isometric to a line segment (as opposed to a circle), and which intersect pair-wise in at most one point. In particular, any vertices of valency $1$ or $\geq3$ of $C$ are in $V(C)$. For each of the $n$ corresponding intervals $e_i$ of $C$, choose a point $y_i$ in its interior and a length parametrisation $[0,l_i]\cong e_i$ (i.e. an orientation for $e_i$). Now choose a real number $\epsilon_i$ such that $(y_i-\epsilon_i,y_i+\epsilon_i)$ lies in the interior of $e_i$, and set $y_i(t):=y_i+t\epsilon_i$ for $t\in(-1,1)$ (see Figure \ref{fig:refined_vertex_set}).

\newcommand{\edgemark}[3]{%
	\pgfmathsetlengthmacro{\crossarm}{1.4*\ptrad}%
	\begin{scope}[shift={(#1,#2)}, rotate=#3]
		\draw[red, line width=1.3pt] (-0.30,0) -- (0.30,0);
		\begin{scope}[rotate=-#3]
			\draw[red, line width=1.3pt] (-\crossarm,-\crossarm) -- (\crossarm,\crossarm);
			\draw[red, line width=1.3pt] (-\crossarm,\crossarm) -- (\crossarm,-\crossarm);
		\end{scope}
\end{scope}}
	
	\begin{figure}[h!]
		\centering
		\begin{tikzpicture}[scale=0.7, >=stealth, line join=round]
			\coordinate (L)  at (-6, 0);
			\coordinate (A)  at (-4, 0);
			\coordinate (Vj) at (-1, 2.0);
			\coordinate (Mv) at (-1, 0.2);
			\coordinate (Bv) at (-1,-1.7);
			\coordinate (R)  at ( 3.4, 0.2);
			\coordinate (x)  at ( 1.64, 0.92);   
			\coordinate (TR) at ( 5.6, 1.3);
			\coordinate (BR) at ( 5.6,-1.3);
			
			\draw[thick] (L)--(A);
			\draw[thick] (A)--(Vj)--(R);      
			\draw[thick] (A)--(Bv)--(R);
			\draw[thick] (Vj)--(Mv)--(Bv);
			\draw[thick] (Mv)--(R);
			\draw[thick] (R)--(TR);
			\draw[thick] (R)--(BR);
			
			\draw[->, thick] (-0.541,-1.502) -- (0.148,-1.205);
			\draw[->, thick] ( 1.433,-0.649) -- (2.076,-0.372);
			
			\foreach \p in {A,Vj,Mv,Bv,x,R} {\filldraw[black] (\p) circle (\ptrad);}
			\pgfmathsetlengthmacro{\sqside}{2*\ptrad}
			\foreach \p in {L,TR,BR} {\fill[blue] (\p) ++(-\ptrad,-\ptrad) rectangle ++(\sqside,\sqside);}
			
			\node[above]       at (Vj) {$v_j$};
			\node[above right]  at (x)  {$x$};
			
			\edgemark{-5.0}{0.0}{0}
			\node[below] at (-5.0,0.0) {$y$};
			\edgemark{-2.5}{1.0}{33.69}
			\edgemark{0.54}{1.37}{-22.25}
			\edgemark{-1.0}{1.1}{90}
			\edgemark{-1.0}{-0.75}{90}
			\edgemark{1.3}{0.2}{0}
			\edgemark{4.5}{0.75}{26.57}
			\edgemark{4.5}{-0.55}{-34.29}
			\edgemark{-2.5}{-0.85}{-29.54}
			\edgemark{0.9}{-0.88}{23.35}
			\node[above left] at (1.24,-1) {$y_i(t)$};
			
			\draw[<->, thick] (-0.731,-2.020) -- (0.948,-1.295);
			\node at (0.245,-2.1) {$e_{i,-}(t)$};
			\draw[<->, thick] ( 1.169,-1.200) -- (3.237,-0.306);
			\node at (2.4,-1.2) {$e_{i,+}(t)$};
		\end{tikzpicture}
		\caption{An example of the vertex sets $V(C)$ and $V(t)$. Points marked by black circles and $1$-valent vertices marked by blue squares form a vertex set $V(C)$. Adding the points marked by red crosses, which can be made to vary along small intervals in their restective edge interiors according to $t\in(-1,1)$, yields the refined vertex set $V(t)$.}
		\label{fig:refined_vertex_set}
	\end{figure}

\def\ptrad{2.6pt}   

	For a given $t\in(-1,1)$, set $V(t):=V(C)\cup\{y_1(t),\dots, y_n(t)\}$, a refined vertex set. We set $r:=\vert V(t)\vert=\vert V(C)\vert +n$, and use the notation $V(t)=\{v_1(t),\dots,v_r(t)\}$ for vertices of $V(t)$, where $v_1(t)=x$, $v_i(t)=v_i$ for $i\leq\vert V(C)\vert$ are the vertices of $V(C)$, and $v_{j+\vert V(C)\vert}=y_j(t)$ for $j\in\{1,\dots,n\}$. 
	Notice the number of segments is now $2n$ rather than $n$ by construction; each $e_i$ has been split into two segments which we denote $e_{i,-}(t)$ and $e_{i,+}(t)$, with lengths $l_{i,-}(t)$ and $l_{i,+}(t)=l_i-l_{i,-}(t)$. With the chosen orientation for $e_i$, these are respectively the segments from its origin to $y_i(t)$ and from $y_i(t)$ to its endpoint (see Figure \ref{fig:refined_vertex_set}). 
	
	The data of a piecewise-linear function $f$ on $C$ whose domain of linearity coincides with the segments $e_{i,\pm}(t)$ is the data of a point in $\R^{r}$, assigning to each vertex $v(t)\in V(t)$ a value $f(v(t))$. When varying $t$, this data yields a family of PL functions $\{f_t\}_{t\in(-1,1)}$ by setting $f_t(v(t))$ to be independent of $t$. Notice that our requirement that the function $f$ we seek is constant on a neighbourhood of each $1$-valent vertex can be implemented as follows.  Let $v_{i_1},\dots,v_{i_m}\in V(C)$ denote the $1$-valent vertices of $C$, and $y_{i_1}(t),\dots,y_{i_m}(t)$ their adjacent vertices in $V(t)$. We impose that $f_t(v_{i_j})=f_t(y_{i_j})$ for any $j=1,\dots,m$. This reduces the input data for $f_t$ to an element of $\R^{r-m}\subset\R^r$. With this in mind, we reduce the set $V(t)$ to $\tilde{V}(t):=V(t)\backslash\{v_{i_1},\dots,v_{i_m}\}$, set $\tilde{r}:=r-m$, re-arrange the indexing so that $\tilde{V}(t):=\{v_1(t),\dots,v_{\tilde{r}(t)}\}=\{x,\dots,y_{1}(t),\dots,y_{n}(t)\}$, and remove the segments adjacent to the $1$-valent vertices from our considerations, leaving us with $2n-m$ segments in the domain of linearity of the rational functions we build. 
	
	Removing the $1$-valent vertices from our considerations also means all remaining segments have finite length. The slope $s_{i,\pm}(f_t)$ of $f_t$ on $e_{i,\pm}(t)$ is given by 
	$$s_{i,\pm}(f_t):=\frac{f_t(v_j(t))-f_t(v_i(t))}{l_{i,\pm}(t)},$$ 
	if $e_{i,\pm}(t)$ is the segment going from $v_i(t)$ to $v_j(t)$. Notice that only the denominator $l_{i,\pm}(t)$ changes when varying $t$ by construction of the functions $f_t$.  
	
	By definition, the coefficient of $\mathrm{div}(f_t)$ at a vertex $v\in V(t)$ is $$\sum_{e_{i,\pm}(t) \text{ adjacent to } v} \pm s_{i,\pm}(f_t),$$ where the sign is $+1$ if $e_{i,\pm}(t)$ is outgoing from $v$, and $-1$ if it terminates at $v$. As found in e.g.\ \cite{bollobas_modern_graph_theory} or \cite[Definition 7]{baker_faber_metrized_graphs}, this can be encored in an $\tilde{r}\times\tilde{r}$-matrix $Q(t)$ with entries
	\begin{align}
		Q_{jk}(t)=\begin{cases}
			-w_{jk} (t)& \text{if $j\neq k$} \\
			\sum_{s=1}^{r}w_{js}(t) & \text{if $j=k$},
		\end{cases} \notag
	\end{align}
	where $w_{jk}(t)=w_{kj}(t)$ is defined as $\frac{1}{l_{i,\pm}(t)}$ if $v_j$ and $v_k$ are connected by a segment $e_{i,\pm}(t)$, and zero if there is no such edge. 
	In these terms, $\mathrm{div}(f_t)=Q(t)\cdot f_t(V(t))$ where $f_t(V(t))=(f_t(v_1(t)),\dots,f_t(v_{\tilde{r}}(t)))$, in the sense that the coefficient of $v_i(t)$ in $\mathrm{div}(f_t)$ is given by the $i^{th}$-coordinate of $Q(t)\cdot f_t(V(t))$. 
	
	We will start by finding a $PL$ function $f_t$ with $\mathrm{div}(f_t)=(k_1+\dots+k_{n})\cdot x-k_1\cdot y_1-\dots-k_{n}\cdot y_{n}$. This amounts to finding a solution to the system $Q(t)\cdot (x_1,\dots,x_{\tilde{r}})=(k_1+\dots+k_n,0,\dots,0,-k_1,\dots,-k_n)$.

	Notice that if $f_t$ is non-constant, then it must achieve its maximum value at a vertex whose coefficient in $\mathrm{div}(f_t)$ is strictly negative (this observation also appears in \cite[Theorem 2]{baker_faber_metrized_graphs} under the name of \enquote{Maximum Principle}). In particular, $\mathrm{div}(f_t)=0$ necessarily implies $f_t$ is constant, which means $\ker(Q(t))$ is $1$-dimensional, so by the rank-nullity Theorem its image is $(r -1)$-dimensional. In fact, it is determined precisely by the codimension-$1$ condition that the divisor has degree zero, so a point in $\R^{\tilde{r}}$ lies in the image of $Q(t)$ if and only if its coordinates sum to zero. In particular, for fixed $t$ and positive integers $k_1,\dots,k_{n}$, there is a function $f_t$ (unique up to a constant) such that $Q(t)\cdot(f_t(v_1(t)),\dots,f_t(v_{\tilde{r}}(t)))=(k_1+\dots+k_n,0,\dots,0,-k_1,\dots,-k_n)$. Recall, such a function is extended to the whole of $C$ by setting its value at the $1$-valent vertices of $C$ to be equal to the value at its adjacent $y_{i_j}(t)$, thereby ensuring it is constant on a neighbourhood of the $1$-valent vertices of $C$, and nowhere else. 
	
	We still need to impose that $f_t$ should have integer slopes; this is where we vary $t$. There are $2n$ slopes $\{s_{i,\pm}(t)\}_{i=1}^n$, each varying strictly monotonically with $t$: the $\{s_{i,-}(t)\}_{i=1}^n$ (resp. $\{s_{i,+}(t)\}_{i=1}^n)$ are strictly decreasing (resp. increasing) with $t$, as the denominators $l_{i,-}(t)$ (resp. $l_{i,+}(t))$ strictly increase (resp. decrease)). Because $\Q\backslash\{0\}$ is dense in $\R$ and because there are finitely many slopes, there exists $t\in(-1,1)$ making all of the $\{s_{i,\pm}(t)\}_{i=1}^n$ rational and non-zero. It now suffices to multiply the corresponding $f_t$ by a sufficiently large integer $N$ so that $\{N\cdot s_{i,\pm}(t)\}_{i=1}^n$ are non-zero integers, and we have $\mathrm{div}(N\cdot f_t)=N\cdot(k_1+\dots+k_n)\cdot x-N\cdot k_1\cdot y_1-\dots-N\cdot k_n\cdot y_s$ as desired.
\end{proof}

\begin{proof}[Proof of Proposition \ref{non-zero_slopes}]
		For $k'>g(C)$, there exists a rational function function $\varphi_{0}:C\rightarrow\R$ with $\mathrm{div}(\varphi_0)=k'\cdot p-q-D_0$ for an effective divisor $D_0$, by \cite[Corollary 6.6]{mikhalkin_Zharkov_trop_curves_jac_theta_funct}. However $\varphi_0$ may well be constant on some of its linearity segments. Using the previous Lemma \ref{non-zero_slopes} for $x=p$, we find a rational function $f$ on $C$ whose slopes are nowhere-zero, and whose divisor is $(N_1+\dots+N_n)\cdot p-N_1\cdot x_1-\dots-N_n\cdot x_n$ for $x_1,\dots, x_n$ interior edge points of $C$ and any positive integers $N_1,\dots N_n$. 
		Taking these integers to be sufficiently large, the function $\varphi:=\varphi_0+f$ is a rational function whose slopes are all non-zero, and $$\mathrm{div}(\varphi)=(k+N_1+\dots+N_n)\cdot p-q-D_0-N_1\cdot x_1-\dots-N_n\cdot x_n,$$ so taking $D=D_0+N_1\cdot x_1+\dots+N_n\cdot x_n$ and $k=k'+N_0+\dots+N_n$ yields the desired result. 
\end{proof}

\begin{remark}
	Notice one could choose $D$ to be supported only on $2$-valent vertices of $C$, by choosing the $N_1,\dots,N_n$ to cancel out the coefficients of higher valency vertices in $D_0$. This could serve to limit the maximal valency of vertices in $\graphvarphi$: for instance if $C$ is trivalent and $p$ and $q$ are also $2$-valent, then $\graphvarphi$ is also trivalent. Although in this paper, we have no reason to limit the maximal valency of $\graphvarphi$, we will make use of this strategy in the follow-up paper \cite{danil_and_me}, as it will facilitate the construction of a Lagrangian lift. With a similar goal in mind, we point out that the function $\varphi$ constructed can be chosen so that no distinct points in the support of $\mathrm{div}(\varphi)$ share the same image. To see this, note first that $p$ must be a (strict) global minimum of $\varphi$, and $q$ a (strict) local maximum, therefore $\varphi(p)>\varphi(q)$. So it remains to impose that the $\varphi(d_i)$ are distinct and different from $\varphi(q),\varphi(p)$ for all points $d_i$ in the support of $D$. This amounts to finding solutions in $\R^{r-m}$ which lie in the complement of finitely many proper hyperplanes in the proof of Lemma \ref{voltage_functions}, which can always be achieved by our perturbation argument. 
\end{remark}

We now construct the family of points $A(t)_{t\in\R}$ as follows. For sufficiently large $k$, choose a rational equivalence $k\cdot p-q-D=\mathrm{div}(\varphi)$ on $C$ as in Proposition \ref{non-zero_slopes}, i.e. $D$ is an effective divisor, $k$ a large enough integer, and $\varphi$ a rational function on $C$ which is constant only on neighbourhoods of its $1$-valent vertices. The balanced graph construction associates to $\varphi$ a tropical $1$-cycle $\graphvarphi$ in $C\times\R$, and points in $A(t)$ are points of $\graphvarphi_t$ counted with a multiplicity we now define.

Let us begin with interior edge points $(x,t)$ of $\graphvarphi$. These fit into two categories:
\begin{enumerate}
	\item $(x,t)$ lies in the interior of an edge $E$ of $\graphvarphi$ which corresponds to the image of a segment of $C$ on which $\varphi$ is linear; 
	\item $(x,t)$ lies in the interior of an edge $E$ of $\graphvarphi$ of the form $\{x\}\times[\varphi(x),+\infty)$ or $\{x\}\times(-\infty,\varphi(x)]$.
\end{enumerate}
If $(x,t)$ is of the first type, we include the point $x\in C$ in the set $A(t)$ with multiplicity equal to the absolute value of the slope of $\varphi$ on the corresponding segment. Note this includes points for which $x$ lies on the interior of segments adjacent to $1$-valent vertices on which $\varphi$ is constant; because their multiplicity is zero, points of such segments do not appear in $A(t)$. We extend this to $1$-valent vertices of $C$, to which we also assign multiplicity $0$ and hence do not appear in $A(t)$.  If $(x,t)$ is of type $(2)$, we include the points $x\in C$ in $A(t)$ with multiplicity equal to the weight of the corresponding edge $E$.

Now, we turn to points $(x,t)$ of $\graphvarphi_t$ of valency $\geq 3$. It will be useful to introduce the notion of \textit{boundary point}:
\begin{definition}
	Let $f$ be a bounded rational function on a compact smooth tropical curve $C$. A \textit{boundary point} of $C$ for $f$ is a point $y\in C$ which is the only endpoint of valency $\geq2$ of $C$ of a segment adjacent to a $1$-valent vertex of $C$ on which $f$ is constant (an example is the point labelled $y$ on the left of Figure \ref{fig:refined_vertex_set}). 
\end{definition}
Notice that by our construction of $\varphi$ above, it follows that each endpoint for $\varphi$ is $2$-valent, and hence adjacent to a \textit{single} $1$-valent vertex $x$ of $C$. It also follows that they lie in the support of $\mathrm{div}(f)$, and $(y,\varphi(y))$ necessarily has valency $3$ in $\graphvarphi$. 

Let $(x,t)$ be any point of $\graphvarphi_t$ of valency $\geq 3$. Its adjacent edges split into \textit{positive}, \textit{negative}, and \textit{null} edges, depending on whether the $\R$-component of their outgoing primitive vector is strictly positive, strictly negative, or zero. By construction, null adjacent edges only occur if $x$ is a boundary point. We include the point $x\in C$ in $A(t)$ with multiplicity equal to the sum of the multiplicities of points arbitrarily close to $(x,t)$ along each of the positive (or negative) edges. Let us verify that this is well-defined. Edges adjacent to $(x,t)$ can be of type $(1)$ or $(2)$ as above. If there are only adjacent edges of type $(1)$, then $x$ corresponds to a point of $C$ of valency $\geq 3$ which is not in the support of $\mathrm{div}(\varphi)$, i.e. the sum of outgoing slopes of $\varphi$ around $x$ is zero, making the multiplicity well-defined. If $(x,t)$ has an adjacent edge of type $(2)$ (note there can only be one), then the weight of that edge is defined precisely to ensure it compensates for the difference in the sums of positive or negative outgoing slopes, ensuring the multiplicity is well-defined. 

Let $\varphi(C)=[a,b]$, and $a=t_0<\dots<t_r=b$ be $\R$-coordinates of points of $\graphvarphi$ of valency $\geq3$.

\begin{lemma}
 	For any given $t\in\R$, $A(t)=\{x_1(t),\dots,x_k(t)\}$ contains $k$ points (with multiplicity), and satisfies the conditions listed in Section \ref{section:results}. Namely, $A(t)$ is well-defined up to a permutation on each connected component of $\R\backslash\{t_0,\dots,t_r\}$, and these permutations can be chosen to make each $x_i(t)\in C$ continuous. Furthermore,  for any $t<a$, $A(t)=\{p,\dots,p\}$, and for any $t>b$, $x_i(t)=q$ for some $i$.
	\label{lemma:A(t)_well_def}
\end{lemma}
\begin{proof}
	Continuity is immediate. Furthermore, by construction of $\graphvarphi$, the only point in $\graphvarphi_t$ for $t<a$ is $\{(p,t)\}$, lying on the interior of the semi-infinite edge $\{p\}\times(-\infty,a]$ which has weight $k$; hence $p$ is the only point of $A(t)$ with multiplicity $k$. Similarly, because $\varphi$ gives the rational equivalence $k\cdot p-q-D\sim0$ in which $q$ appears with coefficient $\leq-1$, then for any $t>b$, $\graphvarphi_t$ contains only points from positive semi-infinite edges of $\graphvarphi$, of which one has the form $\{q\}\times[\varphi(q),+\infty)$ with multiplicity $\geq1$. 
	It remains to show that for any given $t\in\R$, the sums of multiplicities of points occuring in $A(t)$ sum to $k$. Notice that as long as we vary $t\in\R$ continuously such that for all $x_i(t)\in A(t)$, the points $(x_i(t),t)$ remain on the interior of an edge of $\graphvarphi$, the multiplicities with which each $x_i(t)$ appears in $A(t)$ is constant. Hence it suffices to show that varying $t$ continuously in $(t_i-\epsilon,t_i+\epsilon)$ for $t_i\in\{1,\dots,r\}$ does not change the sum of multiplicities of corresponding points in $A(t)$. But the multiplicity of a point $x_j(t_i)$ coming from $(x_j(t_i),t_i)$ of valency $\geq3$ is defined as the sum of multiplicities of nearby points coming from interior edge points of $\graphvarphi$, either on positive or negative edges. So this is equivalent to the fact that these multiplicities are well-defined.
\end{proof}

\begin{example}\label{example_1}
	In Figure \ref{fig:tildeW_type_i_vertex}, we illustrate a portion of $\graphvarphi$ for values of $t$ in $(t_i-\epsilon,t_i+\epsilon)$, where $t_i$ is the $\R$-coordinate of a single $3$-valent vertex of $\graphvarphi$, corresponding to a point $x_1(t_i)$ of $C$ which is \textit{not} in the support of $\mathrm{div}(\varphi)$. The two edges on the right are respectively (from left to right) a semi-infinite vertical edge of weight $w=3$ and the image of a segment of $C$ on which $\varphi$ has slope $s=2$.
		\begin{figure}[h!] 
		\centering
		\begin{tikzpicture}[scale=0.9]
			\draw[axis] (0.5,1) -- (0.5,6) node[above left] {$\mathbb{R}$};
			\draw[axis,<->] (0,1.2) -- (11,1.2) node[right] {$C$};
			
			\draw[level] (0.5,4.0) -- (10.2,4.0) node[right, black] {$t>t_i$};
			\draw[level] (0.5,3.3) -- (10.2,3.3) node[right, black] {$t=t_i$};
			\draw[level] (0.5,2.3) -- (10.2,2.3) node[right, black] {$t<t_i$};
			
			\draw[branch] (3.4,3.3) -- (3.02,2.0) node[below] {$s=-2$};      
			\draw[branch] (3.4,3.3) -- (2.8,5.0) node[above left] {$s=1$};    
			\draw[branch] (3.4,3.3) -- (4.2,5.0) node[above] {$s=1$};   
			
			\node[node] at (3.4,3.3) {};
			\node[right] at (3.55,3.15) {};
			
			\node[dot] at (3.11,2.3) {};  \node[right] at (3.18,2.5) {$(x_1(t),t)$};
			\node[dot] at (3.15,4.0) {};  \node[above] at (2.2,3.85) {$(x_1(t),t)$};
			\node[dot] at (3.73,4.0) {};  \node[above] at (4.67,3.85) {$(x_2(t),t)$};
			
			\draw[branch] (7,2.0) -- (7,4.7) node[above] {$w=3$};
			\node[dot] at (7,3.3) {};
			\node[below right] at (5.14,3.31) {$(x_3(t),t)$};
			
			\draw[branch] (7.85,2.0) -- (9.2,4.7) node[above] {$s=2$};
			\node[node] at (8.5,3.3) {};
			\node[below right] at (8.4,3.31) {$(x_4(t),t)$};
		\end{tikzpicture}
		\caption{An example of a piece of $\graphvarphi$ for $t\in(t_i-\epsilon,t_i+\epsilon)$, where $t_i$ is the $\R$-coordinate of a single $3$-valent vertex of $\graphvarphi$, corresponding to a point $x_1(t_i)$ of $C$ which is not in the support of $\mathrm{div}(\varphi)$.}
		\label{fig:tildeW_type_i_vertex}
	\end{figure}
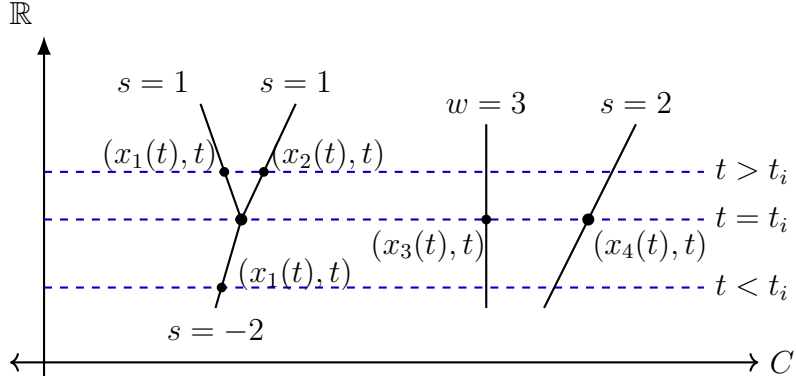
	For $t\in(t_i-\epsilon,t_i)$, we can take $A(t)=\{x_1(t),x_1(t),x_3(t),x_3(t),x_3(t),x_4(t),x_4(t)\}$. The same multiplicities hold for $t=t_i$ as none of the negative edges collide. However, several positive edges branch off after $t=t_i$, because of the trivalent vertex $(x_1(t_i),t_i)$ of $\graphvarphi$. For $t\in(t_i,t_i+\epsilon)$, we can take $A(t)=\{x_1(t),x_2(t),x_3(t),x_3(t),x_3(t),x_4(t),x_4(t)\}$. 
\end{example}

\begin{example}
	In the example represented in Figure \ref{fig:example_2}, we vary $t$ across $t_i\in\{t_0,\dots,t_r\}$ which is the $\R$-coordinate of a boundary point for $\varphi$, $y\in C$, as well as a trivalent vertex $(x,t_i)$ of $\graphvarphi$ for which $x$ is in the support of $\mathrm{div}(\varphi)$. 
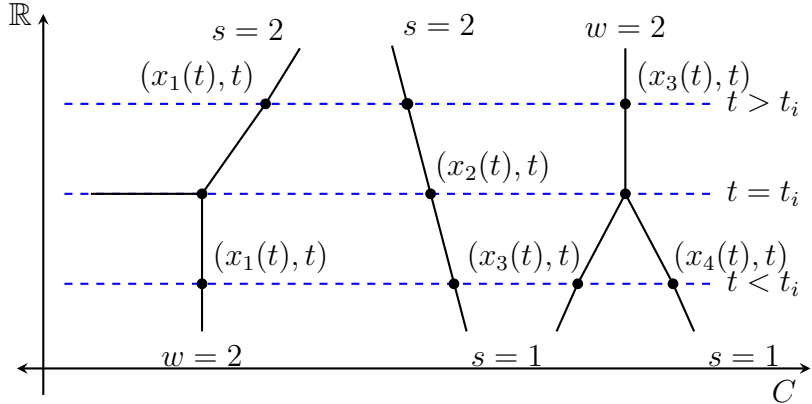
\begin{figure}[h!]
	\centering
	\begin{tikzpicture}[scale=0.70, >=stealth, line join=round]
		\draw[->, thick] (0.5,0.4) -- (0.5,7.6) node[left] {$\mathbb{R}$};
		\draw[<->, thick] (0.0,0.9) -- (15.0,0.9);
		\node[below] at (14.5,0.88) {$C$};
		
		\def\ybot{2.5} \def\ymid{4.2} \def\ytop{5.9}
		\draw[blue, dashed, thick] (0.9,\ytop) -- (13.2,\ytop) node[right, black]{$t>t_i$};
		\draw[blue, dashed, thick] (0.9,\ymid) -- (13.2,\ymid) node[right, black]{$t=t_i$};
		\draw[blue, dashed, thick] (0.9,\ybot) -- (13.2,\ybot) node[right, black]{$t<t_i$};
		
		\coordinate (V1) at (3.5,\ymid);
		\coordinate (d2) at (4.7,\ytop);
		\coordinate (e1) at (3.5,\ybot);
		\draw[thick] (1.4,\ymid) -- (V1);                
		\draw[thick] (V1) -- (d2) -- (5.35,6.95);         
		\draw[thick] (V1) -- (e1) -- (3.5,1.6);           
		\node[above left] at (5.25,6.9) {$s=2$};
		\node[below] at (3.5,1.55) {$w=2$};
		\node[above left] at (d2) {$(x_{1}(t),t)$};
		\node[above right] at (e1) {$(x_{1}(t),t)$};
		
		\coordinate (mt) at (7.38,\ytop);   
		\coordinate (m2) at (7.82,\ymid);   
		\coordinate (mb) at (8.26,\ybot);   
		\draw[thick] (7.09,7.0) -- (8.50,1.6);
		\node[above right] at (7.09,7.0) {$s=2$};
		\node[above right] at (7.7,\ymid) {$(x_2(t),t)$};
		
		\coordinate (V3) at (11.5,\ymid);
		\coordinate (d3) at (11.5,\ytop);
		\coordinate (e3) at (10.6,\ybot);
		\coordinate (e4) at (12.4,\ybot);
		\draw[thick] (V3) -- (d3) -- (11.5,6.95);         
		\draw[thick] (V3) -- (e3) -- (10.2,1.6);          
		\draw[thick] (V3) -- (e4) -- (12.8,1.6);          
		\node[above] at (11.5,6.95) {$w=2$};
		\node[below left]  at (10.15,1.55) {$s=1$};
		\node[below right] at (12.85,1.55) {$s=1$};
		\node[above right] at (d3) {$(x_{3}(t),t)$};
		\node[above left]  at (10.89,\ybot) {$(x_{3}(t),t)$};
		\node[above right] at (12.2,\ybot) {$(x_{4}(t),t)$};
		
		\foreach \p in {V1,d2,e1,m2,mb,V3,d3,e3,e4} {\filldraw (\p) circle (2.6pt);}
		\draw[fill=black, thick] (mt) circle (2.6pt);     
	\end{tikzpicture}
	\caption{An example of a piece of $\graphvarphi$ for $t\in(t_i-\epsilon,t_i+\epsilon)$, where $t_i$ is the $\R$-coordinate of a boundary point for $\varphi$ as well as a trivalent vertex $(x,t_i)$ of $\graphvarphi$ for which $x$ is in the support of $\mathrm{div}(\varphi)$.}
	\label{fig:example_2}
\end{figure}
	For $t<t_i$, $\graphvarphi_t$ contains four points: from left to right, $x_1(t)$ lies on the interior of a semi-infinite vertical edge, $x_2(t)$ lies on the interior of an edge which is the image of a segment of $C$ on which $\varphi$ has slope $\pm2$, and $x_3(t)$ and $x_4(t)$ lie on the interior of adjacent edges which come from (adjacent) segments of $C$ on which $\varphi$ has slope $\pm1$. Therefore, we can take $A(t)=\{x_1(t),x_1(t),x_2(t),x_2(t),x_3(t),x_4(t)\}$. Similarly, for $t=t_i$ we can take $A(t_i)=\{x_1(t_i),x_1(t_i),x_2(t_i),x_2(t_i),x_3(t_i),x_3(t_i))\}$, and $x_3(t_i)=x_4(t_i)$ causes a change in the multiplicities appearing in $A(t)$. Finally, for $t>t_i$, we can take $A(t)=\{x_1(t),x_1(t),x_2(t),x_2(t),x_3(t),x_3(t)\}$, and this time $x_3(t)=x_3$ lies on a vertical semi-infinite edge. 
\end{example}

\begin{remark}\label{remark:increasing_k}
	We will see that procedure of increasing $k$ to ensure \enquote{transversality} through Lemma \ref{non-zero_slopes} will correspond to increasing the smallest $k$ for which $k!\cdot(p-q)^k$ is rationally trivial, and therefore weakening the result. In fact, we expect the result to hold for any $k>g(C)$, even if our proof does not provide this optimal bound. The reason we have made this choice, is that allowing $\varphi$ to have zero slopes on a segment of $C$ means its interior points would not appear in $A(t)$. When constructing the rational equivalence $W\subset C^k\times\R$ from $A(t)$, this will obstruct balancing, and edges will need to be added manually to compensate for this. This is what we will do for the semi-infinite edges on which $\varphi$ is constant, and makes the construction of $W$ more involved. The same procedure for interior segments of $C$ would be significantly more cumbersome. 
\end{remark}

\begin{example}
	This final example represents the full two-ended balanced graph of a rational function $f$ on a circle or circumference $a$, $S^1_A\cong\R/a\Z$, whose divisor is $\mathrm{div}(f)=3\cdot p-q-2\cdot y$. This is the function described in Example \ref{ex:increase_k_get_non-zero_slopes}; applying to it our construction below yields $3!\cdot(p-q)^3\sim0$ on $S^1_a$, which illustrates Remark \ref{remark:increasing_k} above since we prove in Section \ref{section:examples_1} that while there can be no rational function $\varphi$ on $S^1_a$ whose slopes are nowhere-zero and whose divisor is $\mathrm{div}(\varphi)=2\cdot p-q-y$ for some $y\in S^1_a$, it is nevertheless true that $2!\cdot (p-q)^2\sim0$. 
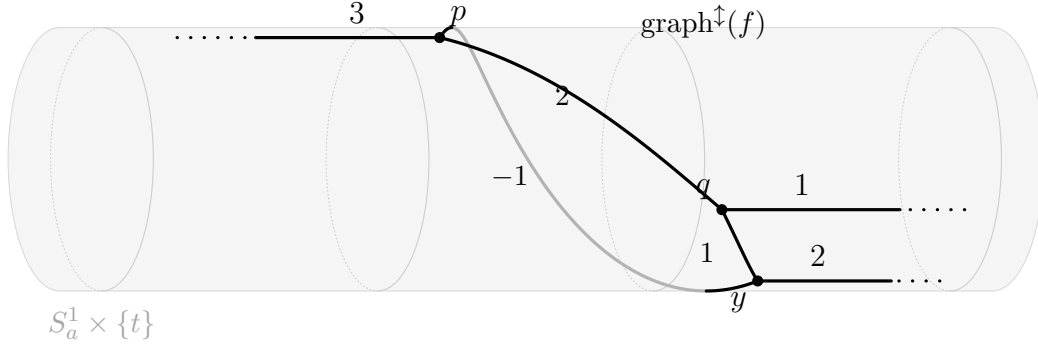
\begin{figure}
\begin{tikzpicture}[>=Stealth,line cap=round,scale=0.85]
	
	\tikzset{dotted end/.style={very thick,line cap=round,dash pattern=on 0pt off 5pt}}
	
	\pgfmathsetmacro{\Sa}{12}                 
	\pgfmathsetmacro{\LI}{3}                  
	\pgfmathsetmacro{\LII}{(\Sa-3*\LI)/2}     
	\pgfmathsetmacro{\LIII}{(\Sa+\LI)/2}      
	\pgfmathsetmacro{\Sq}{\LI}                
	\pgfmathsetmacro{\Sy}{\LI+\LII}           
	\pgfmathsetmacro{\Fq}{2*\LI}              
	\pgfmathsetmacro{\Fy}{\LIII}              
	
	\pgfmathsetmacro{\Ry}{2.05}   
	\pgfmathsetmacro{\Rx}{0.80}   
	\pgfmathsetmacro{\Kx}{0.66}   
	\pgfmathsetmacro{\Rot}{22.5}  
	\pgfmathsetmacro{\Deg}{360/\Sa}
	\pgfmathsetmacro{\Xmin}{-8.5} 
	\pgfmathsetmacro{\Xmax}{13.5} 
	\pgfmathsetmacro{\Shid}{(180-\Rot)/\Deg}   
	\pgfmathsetmacro{\Svis}{(360-\Rot)/\Deg}   
	
	\def\CYL#1#2{({(#2)*\Kx + \Rx*sin(\Rot+\Deg*(#1))},{\Ry*cos(\Rot+\Deg*(#1))})}
	\def\Ff#1{ifthenelse(#1<\Sq, 2*(#1), ifthenelse(#1<\Sy, \Fq+(#1-\Sq), \Fy-(#1-\Sy)))}
	\def\GRAPHARC#1#2{plot[domain=#1:#2,samples=80,variable=\s,smooth]
		({(\Ff{\s})*\Kx + \Rx*sin(\Rot+\Deg*\s)},{\Ry*cos(\Rot+\Deg*\s)})}
	
	\path \CYL{0}{0}         coordinate (p);
	\path \CYL{\Sq}{\Fq}     coordinate (q);
	\path \CYL{\Sy}{\Fy}     coordinate (y);
	\path \CYL{0}{-6.2}      coordinate (pend);
	\path \CYL{\Sq}{\Fq+6.0} coordinate (qend);
	\path \CYL{\Sy}{\Fy+4.5} coordinate (yend);
	\path \CYL{1.5}{3}       coordinate (m1);
	\path \CYL{3.75}{6.75}   coordinate (m2);
	\path \CYL{8}{4}         coordinate (m3);
	
	\fill[gray!8]  ({\Xmin*\Kx},\Ry) -- ({\Xmax*\Kx},\Ry)
	arc[start angle=90,end angle=-90,x radius=\Rx,y radius=\Ry]
	-- ({\Xmin*\Kx},-\Ry)
	arc[start angle=-90,end angle=-270,x radius=\Rx,y radius=\Ry] -- cycle;
	\draw[gray!30] ({\Xmin*\Kx},\Ry) -- ({\Xmax*\Kx},\Ry)
	arc[start angle=90,end angle=-90,x radius=\Rx,y radius=\Ry]
	-- ({\Xmin*\Kx},-\Ry)
	arc[start angle=-90,end angle=-270,x radius=\Rx,y radius=\Ry];
	
	\foreach \t in {-7.5,-1,5.5,12.5}{%
		\draw[gray!40]                ({\t*\Kx},\Ry) arc[start angle=90,end angle=-90,x radius=\Rx,y radius=\Ry];
		\draw[gray!40,densely dotted] ({\t*\Kx},\Ry) arc[start angle=90,end angle=270,x radius=\Rx,y radius=\Ry];
	}
	\node[gray!70,font=\small,below=2pt] at ({-7.5*\Kx},-\Ry) {$S^1_a\times\{t\}$};
	
	
	\draw[very thick,gray!60] \GRAPHARC{\Shid}{\Svis};
	\draw[very thick] \GRAPHARC{0}{\Shid};
	\draw[very thick] \GRAPHARC{\Svis}{\Sa};
	
	\draw[very thick] (p) -- ($(p)!0.70!(pend)$) node[pos=0.45,above=1pt] {$3$};
	\draw[dotted end] ($(p)!0.74!(pend)$) -- (pend);
	\draw[very thick] (q) -- ($(q)!0.70!(qend)$) node[pos=0.45,above=1pt] {$1$};
	\draw[dotted end] ($(q)!0.74!(qend)$) -- (qend);
	\draw[very thick] (y) -- ($(y)!0.70!(yend)$) node[pos=0.45,above=1pt] {$2$};
	\draw[dotted end] ($(y)!0.74!(yend)$) -- (yend);
	
	\draw[fill=black] (p) circle (2.2pt) node[above right] {$p$};
	\draw[fill=black] (q) circle (2.2pt) node[above left]  {$q$};
	\draw[fill=black] (y) circle (2.2pt) node[below left]  {$y$};
	
	\node[font=\small] at ($(m1)+(-0.5,0.2)$) {$2$};
	\node[font=\small] at ($(m2)+(-0.55,0)$)  {$1$};
	\node[font=\small] at ($(m3)+(-0.45,0)$)  {$-1$};
	
	\node[font=\small,anchor=west] at ($(m1)+(0.55,1.35)$) {$\graphf$};
	
\end{tikzpicture}
	\caption{The two-ended balanced graph, $\graphf$, of the function $f$ described in Example \ref{ex:increase_k_get_non-zero_slopes} on $S^1_a$ (with $a=l_1+l_2+l_3$). The $\R$-direction is horizontal. All bounded edges have weight $1$, while semi-infinite edges carry the weights $3$, $1$, $2$ of the divisor $\mathrm{div}(f)=3\cdot p-q-2\cdot y$.}
	\label{fig:two_ended_balanced_graph_full_example}
\end{figure}
\end{example}

\subsection{Construction of the rational equivalence}\label{section:constructing_W}

We now build a $1$-cycle $W\subset C^k\times\R$ which realises a rational equivalence $k!\cdot(p-q)^k\sim0$ in $C^k$. However, by construction this cycle will naturally be the closure of a cycle $W^{[0]}\subset(C^{[0]})^k\times\R$, and therefore the rational equivalence $k!\cdot(p-q)^k\sim0$ can be realised in $(C^{[0]})^k$. Note the boundary divisor of $C^k\times\R$ is given by $$\bigcup_{i=1}^kC\times\dots\times \underbrace{C_\infty}_{i^{th}}\times\dots\times C \times\R,$$ where $C_\infty=C\backslash C^{[0}$ are just the $1$-valent vertices of $C$. For each such $1$-valent vertex $x$ with corresponding endpoint $y$, and each of the $k$ components of the boundary divisor above, there is a neighbourhood of the form $$C\times\dots\times \underbrace{[x,y)}_{i^{th}}\times\dots\times C \times\R$$ of $C\times\dots\times\{x\}\times\dots\times C \times\R$ in which $W$ is of the form $$\{z_1\}\times\dots\times\underbrace{[x,y)}_{i^{th}}\times\dots\times\{z_k\}\times\{t\}$$ with $\{z_1,\dots,z_k\}$ points of $C$ and $t\in\R$. In particular, this is the closure of the corresponding piece of $W^{[0]}$ of the form $\{z_1\}\times\dots\times(x,y)\times\dots\times\{z_k\}\times\{t\}$.

  The starting point is the one-parameter family of points $A(t)=\{x_1(t),\dots,x_k(t)\}$ from the previous Section, from which we reuse notation.

Recall $W$ is defined by its fibre $W_t:=(\pi_{C^k})_*(W\cdot(C^k\times\{t\}))$ above any $t\in\R$:
\begin{align}
	W_t=\bigoplus_{\sigma\in S_k} (x_{\sigma(1)}(t)-q,\dots,x_{\sigma(k)}(t)-q) \subset C^k.
	\label{eq:fibres_of_W}
\end{align}

Observe that when $t<a$, then $A(t)=\{p,\dots,p\}$, and for $t>b$, then $x_i(t)=q$ for some $i\in\{1,\dots,k\}$. This means that  $$W_t=\begin{cases}
	k!\cdot(p-q)^k & \text{if $t\leq a$} \\
	0 & \text{if $t\geq b$}.
\end{cases}$$

In this Section, we study the underlying set $\vert W\vert$ of $W$ in $C^k\times\R$ and show:
\begin{proposition}\label{prop:W_tropical_structure}
	$\vert W \vert\subset C^k\times\R$ carries a natural structure of a tropical cycle, which we denote $W$.
\end{proposition}
Together with the observation above, this implies that $W$ witnesses a rational equivalence $k\cdot(p-q)^k\sim0$, proving Proposition \ref{prop:tropical_voevodsky_curves}.

\subsubsection{Edges of $W $}

To write $\vert W\vert$ as a union of linear segments, recall from the previous Section that as long as $t$ varies along a connected component of $\R\backslash\{t_1,\dots,t_r\}$, each $x_i(t)$ appears in $A(t)$ with constant multiplicity, and simply varies along a given segment of $\graphvarphi$. Consequently, corresponding points of $W$ are also interior edge points, and we will see the $t_i$'s will also be precisely the $\R$-coordinates of higher-valency vertex points of $W$.

For every $\sigma\in S_k$, the corresponding summand in Equation \eqref{eq:fibres_of_W} expands into $2^k$ terms which we will label by $\{i_1,\dots,i_k\}\in\{0,1\}^k$. We introduce the notation
$$z_{(i_1,\dots,i_k)}(t):=(z^{i_1}_1(t),\dots,z^{i_{k}}_k(t))\in C^k,$$ where $z_i^0(t)=x_i(t)$, and $z^1_i(t)=q$. These correspond to the $C^k$-components of points of $W$. 

The first edges of $\vert W\vert$ we describe are the semi-infinite edges with $\R$-coefficients in $(-\infty;a]$, which are of the form
\begin{align}
	N_{\{i_1,\dots,i_k\},\sigma}:=\sigma(z_{(i_1,\dots,i_k)})\times(-\infty,a]\subset C^k\times\R, \notag
\end{align}
Recall that $A(t)=\{p,\dots,p\}$ for $t<a$, hence $z_i^0=p$ and $z_i^1=q$ for all $i\in\{1,\dots,k\}$. There are $k!\cdot2^k$ such edges in total, with the $k!$-factor coming from the different $\sigma\in S_k$, and $2^k$ from the binomial expansion for a given $\sigma$. For any fixed $\sigma\in S_k$, there are $\binom{k}{k_0}$ (distinct) edges $N_{\{i_1,\dots,i_k\},\sigma}$ for which $0$ appears $k_0$ times in the tuple $\{i_1,\dots,i_k\}$. This means that there are $k!\cdot\binom{k}{k_0}$ such edges in total, with $\binom{k}{k_0}$ \textit{distinct} ones, each appearing with multiplicity $k!$. When considering the cycle structure on $\vert W \vert$, each of these segments will be assigned a sign $(-1)^{(k-k_0)}$.

Similarly, semi-infinite edges of $\vert W\vert$ are those with $\R$-values in $[b;+\infty)$, of the form
\begin{align}
	P_{\{i_1,\dots,i_k\},\sigma}:=\sigma(z_{(i_1,\dots,i_k)})\times[b;+\infty)\subset C^k\times\R. \notag
\end{align}
Recall that for $t> b$, $A(t)$ is, up to permutation, $\{q, d_1,\dots,d_{k-1}\}$, where $D=d_1+\dots+d_{k-1}$ is the effective divisor appearing in the rational equivalence $k\cdot p-q-D=\mathrm{div}(\varphi)$. Therefore we can take $z_1^0=q$ and $z_i^0=d_{i-1}$ for $i\in\{2,\dots,k\}$. The number of \textit{different} such edges, and the multiplicities with which they appear, depend on the combinatorics of $\{q, d_1,\dots,d_{k-1}\}$. Again, we can anticipate assigning to each segment indexed by a tuple $\{i_1,\dots,i_k\}$ in which $0$
 appears $k_0$ times a sign $(-1)^{k-k_0}$, which will become relevant when considering the cycle structure on $W $.

Edges of $\vert W\vert$ with $\R$-coordinates in $[t_{i},t_{i+1}]$ are of the form
\begin{align}
	F_{\{i_1,\dots,i_k\},\sigma}^{[t_i,t_{i+1}]}:=\{\sigma(z_{(i_1,\dots,i_k)}(t),t)\}_{t\in[t_i,t_{i+1}]}\subset C^k\times\R \notag
\end{align}
for any $\{i_1,\dots,i_k\}\subset\{0,1\}^k$, where $z_i^0(t)=x_i(t)$ and $z_i^1(t)=q$. 

Although the edges described above are the only ones whose points explicitly appear in Equation \eqref{eq:fibres_of_W}, there would be no chance for $\vert W\vert$ to be balanced as such. To see this, notice that boundary points $y$ for $\varphi$ appear with non-zero multiplicity in $A(t)$ for $t\geq\varphi(y)$ or $t\leq\varphi(y)$, depending on the sign of the non-zero coefficient of $y$ in $\mathrm{div}(\varphi)$. However, points in the leaf $[x,y)$, where $x$ is the $1$-valent vertex of $C$ adjacent to $y$, do not appear in $A(t)$ for any $t\in\R$. Edges corresponding to these leaves need to manually be inserted in $\vert W\vert$ for balancing to be possible. These do not feature in Equation \eqref{eq:fibres_of_W} because they are contained in $C^k\times\{\varphi(y)\}$, making them invisible to the divisor in $W_{\varphi(y)}:=(\pi_{C^k})_*(\mathrm{div}(\pi^*_{C^k\times\R}\max(t,\varphi(y)))\vert_W)$ (see \cite[Construction 3.3, Definition 3.4]{tropical_intersection_theory_allerman_rau}).

To construct these additional edges for $\vert W\vert$, let $m(y)$ be the multiplicity with which $y$ appears in $A(\varphi(y))$. This means that a point $\tilde{v}$ in $W_{\varphi(y)}$ can have up to $m(y)$ copies of $y$ as $C$-coordinates. For a fixed $\tilde{v}\in W_{\varphi(y)}$, we denote by $\vert y\vert$ this number. If $\vert y\vert\geq1$, we manually insert $\vert y\vert$ adjacent edges to $\tilde{v}$: one for each of its $y$-coordinates. Because $\tilde{v}$ is of the form $\sigma(z_{(i_1,\dots,i_k)}(\varphi(y),\varphi(y)))$ for some $\sigma\in S_k$ and $\{i_1,\dots,i_k\}\in\{0,1\}^k$, each $y$-coordinate corresponds to an index $j$ such that $i_{\sigma(j)}=0$ and $x_{\sigma(j)}(\varphi(y))=y$. Then the corresponding adjacent edge we add has the form:
\begin{align}
	L_{\{i_1,\dots,i_k\},\sigma,j}:= \sigma\left(z^{i_1}_{1}(\varphi(y))\times\dots\times \underbrace{[x,y]}_{j^{th}}\times\dots\times z^{i_k}_k(\varphi(y))\right)\times\{\varphi(y)\}. \notag 
\end{align}

It is important to note that the multi-index $(\{i_1,\dots,i_k\},\sigma,j)$ does \textit{not} mean we associate such an edge to \textit{each} $\{i_1,\dots,i_k\}\in\{0,1\}^k$, but only to those $2^{k-m(y)}\cdot(2^{m(y)}-1)$ tuples for which there is at least one $y$-coordinate in $z_{(i_1,\dots,i_k)}(\varphi(y))$. Let these be indexed by $\Theta\subset\{0,1\}^k$. Furthermore, the range of the index $j$ depends on the chosen tuple in $\Theta$, namely it ranges from $1$ to the number $\vert y\vert$ of $y$-coordinates of $z_{(i_1,\dots,i_k)}(\varphi(y))$. To avoid cumbersome notation, for any endpoint $y$ of $\varphi$, we simply write 
\begin{align}
	L(y):=\bigcup_{\substack{\{i_1,\dots,i_k\}\in\Theta\\ \sigma\in S_k\\ j\in\{1,\dots,\vert y\vert\}}} L_{\{i_1,\dots,i_k\},\sigma,j}
	\subset C^k\times\{\varphi(y)\}
	\notag
\end{align}
for the union over all such \enquote{admissible} indices $\{i_1,\dots,i_k\},\sigma,j$.

With this notation, we can write the decomposition of $\vert W\vert$ into linear segments as follows: 
\begin{align}
	\vert W\vert=
	\bigcup_{\substack{\{i_1,\dots,i_k\}\in\{0,1\}^k\\ \sigma\in S_k}} 
	\left(
	N_{\{i_1,\dots,i_k\},\sigma}
	\cup 
	\left(\bigcup_{l=0}^{r-1}
	F_{\{i_1,\dots,i_k\},\sigma}^{[t_i,t_{i+1}]}
	\right)
	\cup
	P_{\{i_1,\dots,i_k\},\sigma}
	\right)
	\bigcup_{\substack{y \hspace{0.5mm}\mathrm{endpoint}\\ \mathrm{for}\hspace{0.5mm}\varphi}} L(y).
	\label{setwise_rat_equiv}
\end{align}
Notice each segment is rational in $C^k\times\R$ with respect to the tropical structure on $C$ and the standard tropical structure $\Z\subset\R$ on $\R$, because $\varphi$ has integer slopes.

\subsubsection{Weights of $W$}\label{section:weights_of_W}

Equation \eqref{setwise_rat_equiv} describes the set $\vert W\vert$. To equip it with the structure of a tropical cycle, we must assign an integer weight to each of the segments $N_{\{i_1,\dots,i_k\},\sigma}$, $F_{\{i_1,\dots,i_k\},\sigma}$, $P_{\{i_1,\dots,i_k\},\sigma}$, and $L(y)$.
It is important to distinguish the weight of an edge with its \textit{multiplicity}, which just records its (signed) occurences in $\vert W\vert$.

\begin{definition}[Multiplicity]\label{def:multiplicity_of_segment}
	Let $E$ be one of the segments of Equation \eqref{setwise_rat_equiv} of type $N_{\{i_1,\dots,i_k\},\sigma}$, $F_{\{i_1,\dots,i_k\},\sigma}$ or $P_{\{i_1,\dots,i_k\},\sigma}$, and $t\in\mathrm{int}(\pi_\R(E))\in\R$ such that $E_t:=E\cap(C^k\times\{t\})$ is a single point. 
	We define the \emph{multiplicity} $\mu(E)$ of $E$ as
	$$\mu(E):=\sum_{\substack{(\sigma,(i_1,\dots,i_k))\in S_k\times\{0,1\}^k \\ \sigma(z_1^{i_1}(t),\dots,z_k^{i_k}(t),t)=E_t}}(-1)^{k-k_0}\in\Z,$$
	where as before $k_0$ denotes the number of zeroes in $(i_1,\dots,i_k)$.
\end{definition}

\begin{remark}\label{rmk:sign_factors_out}
	These multiplicities behave qualitatively differently for $t<\varphi(q)$ or $t>\varphi(q)$. Because $q$ is in the support of $\mathrm{div}(\varphi)$ with coefficient $\leq-1$, $\graphvarphi$ has a semi-infinite edge $\{q\}\times[\varphi(q);+\infty)$ of weight $\geq1$. This means that $q\in A(t)$ if and only if $t\geq\varphi(q)$. Notice that, as soon as $x_i(t)=q$ for some $i$, then $$(x_{\sigma(1)}(t)-q,\dots,x_{\sigma(k)}(t)-q)=0,$$ so $\mu(E)=0$ for \textit{any} edge with $\R$-coordinates $t\geq\varphi(q)$. On the other hand, if $t<\varphi(q)$, then a coordinate of $E_t$ equal to $q$ can only arise from $i_j=1$, so all pairs $(\sigma,(i_1,\dots,i_k))$ contributing to $\mu(E)$ share the same tuple $(i_1,\dots,i_k)$ -- namely the one with $i_j=1$ exactly at the $q$-coordinates of $E_t$ -- and in particular the same sign $(-1)^{k-k_0}$, which may thus be be factored out of the sum in Definition \ref{def:multiplicity_of_segment} .
\end{remark}

\begin{lemma}[Formula for $\mu(E)$]\label{lemma:closed_formula_multiplicity}
	Let $E$ and $t$ be as in Definition \ref{def:multiplicity_of_segment}, and assume $q\notin A(t)$ (i.e. $t<\varphi(q)$). Denote by $v_1,\dots,v_l\in C$ the distinct values taken by the $k_0$ coordinates of $E_t$ different from $q$, by $n_j$ the number of coordinates of $E_t$ equal to $v_j$, and by $m_j$ the multiplicity of $v_j$ in $A(t)$ (so $n_j\leq m_j$ and $n_1+\dots+n_l=k_0$). Then $$\mu(E)=(-1)^{k-k_0}\cdot(k-k_0)!\cdot\prod_{j=1}^{l}\frac{m_j!}{(m_j-n_j)!}.$$
\end{lemma}
\begin{proof}
	Since no $x_i(t)$ equals $q$, the tuple $(i_1,\dots,i_k)$ is determined by the $k-k_0$ coordinates of $E_t$ equal to $q$.  Then $\sigma$ matches the $n_j$ coordinates holding the value $v_j$ with $n_j$ \emph{distinct} indices among the $m_j$ indices $i$ with $x_i(t)=v_j$: there are $m_j!/(m_j-n_j)!$ ways of doing so for each $j$, all independent. The remaining $k-k_0$ indices may be distributed freely over the $q$-coordinates, in $(k-k_0)!$ ways.
\end{proof}

\begin{example}
	Applying this formula to a positive edge from Example \ref{example_1} with $k=k_0$, one gets points of the form $\sigma(x_1(t),x_2(t),x_3(t),x_3(t),x_3(t),x_4(t),x_4(t),t)$. Because $k=k_0$, $n_j=m_j$ for every $j$, and the multiplicity is $3!\cdot 2!=12$, obtained by different permutations of the $x_3(t)$'s and $x_4(t)$'s. Now the edge with points of the form $\sigma(q,x_2(t),x_3(t),q,q,x_4(t),x_4(t))$, for which $k_0=k-3$, has multiplicity $-3!\cdot\frac{3!}{2!}\cdot 2!=36$. The $3!$-term corresponds to permuting the $q$-factors, the $\frac{3!}{2!}$-term corresponds to the three copies of $x_3(t)$ one can \enquote{choose from} in $A(t)$ (i.e. the three different indices $j_1,j_2,j_3$ for which $i_j=0$ contributes an $x_3(t)$-coordinate), and the $2!$-factor corresponds to permuting the $x_4(t)$-factors. 
\end{example}

We now want to assign to every edge $E$ in $W$ a \textit{weight} $w(E)$; while the multiplicity $\mu(E)$ might be a natural candidate, Proposition \ref{prop:weights_well_defined} will demonstrate this is too naive. Recall $W$ is constructed so that its fibres $W_t:=(\pi_{C^k})_*(W\cdot \mathrm{div}(\pi_\R^*{\max(x,t)}))$ are as in Equation \eqref{eq:fibres_of_W}. The subtlety comes from how the weights of the Weil divisor $W\cdot \mathrm{div}(\pi_\R^*{\max(x,t)})$ associated to the Cartier divisor $\mathrm{div}(\pi_\R^*{\max(x,t)})$ are defined, following \cite[Construction 3.3, Definition 3.4]{tropical_intersection_theory_allerman_rau}.
Namely, they only inherit the weight of the corresponding edge in $W$ up to multiplication by the $\R$-coordinate (which is $\Z$-valued) of the primitive tangent vector to $E$:

\begin{definition}[Primitive tangent vector to $E$]
	Let $E$ and $t$ be as in Definition \ref{def:multiplicity_of_segment}, and orient $E$ so that the $\R$-coordinate is increasing. Let $u_j$ denote the primitive tangent vector, in the direction of increasing $t$, to the segment of $C$ on which the $j^{th}$ coordinate of the point $E_t\in C^k$ lies. Notice that $u_j=0$ in exactly two settings: if $i_j=1$, or if $i_j=0$ and the corresponding $x_i(t)$ comes from a semi-infinite edge of $\graphvarphi$. If $u_j\neq0$, we denote by $s_j$ the slope of $\varphi$ on the segment corresponding segment of $C$. Otherwise, we set $s_j=1$.
	Setting $c(E):=\operatorname{lcm}(\vert s_j\vert)$, the \textit{primitive tangent vector} of $E$ is $$u_E=\left(\frac{c(E)\cdot u_1}{\vert s_1\vert},\dots,\frac{c(E)\cdot u_k}{\vert s_k\vert},c(E)\right)\in\Lambda_C\times\dots\times\Lambda_C\times\Z,$$ where $\Lambda_C$ denotes the (local) lattice determining the tropical structure on $C$. 
	\label{def:primitive_vector}
\end{definition}

\begin{definition}[Weights for edges $N_{\{i_1,\dots,i_k\},\sigma}$, $F_{\{i_1,\dots,i_k\},\sigma}$ and $P_{\{i_1,\dots,i_k\},\sigma}$]\label{def:weights_of_W_1}
	The \emph{weight} of a segment $E$ of $\vert W\vert$ of type $N_{\{i_1,\dots,i_k\},\sigma}$, $F_{\{i_1,\dots,i_k\},\sigma}$ or $P_{\{i_1,\dots,i_k\},\sigma}$ 
	is $$w(E):=\frac{\mu(E)}{c(E)}.$$
\end{definition}
	
\begin{remark}
	Notice these are integers: by Lemma \ref{lemma:closed_formula_multiplicity}, $\mu(E)$ is divisible by $\prod_j \frac{m_j!}{(m_j-n_j)!}$, where $j$ ranges over the non-constant coordinate values of $E$, so that each factor is divisible by $m_j$. Because these non-constant values are interior edge points of $C$, $m_j=\vert s_j\vert$. Hence $\mu(E)$ is divisible by $\prod_j\vert s_j\vert$, which is in turn divisible by $c(E)=\operatorname{lcm}_j\{\vert s_j \vert\}$.
\end{remark}
	
	It remains to define weights for edges of $\vert W\vert$ of the form $L(y)$, where $y\in C$ is an endpoint for $\varphi$. In this case, we cannot infer these from our requirement that the fibres $W_t$ satisfy Equation \eqref{eq:fibres_of_W}. This is because $\pi_\R^*{\max(x,t)}$ is constant on these edges, hence they are automatically assigned zero coefficient in $\mathrm{div}(\pi_\R^*{\max(x,t)})$ (consistently with the fact that the corresponding points have multiplicity zero in $A(t)$). However, as the example below demonstrates, their weights are prescribed by the requirement that $\vert W\vert$ is balanced:

\begin{example}[Balancing using $L(y)$]\label{example:weights_of_leaf_edges}
	Let $e_y$ be a segment of $C$ adjacent to a $1$-valent vertex on which $\varphi$ is constant equal to $t_i$, and let $y$ be its $2$-valent endpoint. Suppose $\varphi$ has slope $s$ on the interior edge adjacent to $y$, so that the coefficient of $y$ in $\mathrm{div}(\varphi)$ is $-s$, and $\graphvarphi$ contains the vertical edge $\{y\}\times[t_i,+\infty)$ with weight $s$ (Figure \ref{fig:leaf_edge_local_structure}). Assuming $k=s$, $A(t)$ consists of a point $x(t)$ of multiplicity $s$ ascending to $y$ for $t<t_i$, and of the constant point $y$ with multiplicity $s$ for $t\geq t_i$. Let $u$ be the primitive direction at $y$ pointing into the interior edge, so that $-u$ points into $e_y$.
	
	The vertex $\tilde{v}=(y,\dots,y,t_i)$ of $W$ has two monotone adjacent edges. The first is the negative edge $E^-=\{(x(t),\dots,x(t),t)\}_{t<t_i}$ with $\mu(E^-)=s!$, and outgoing primitive vector $(u,\dots,u,-s)$ hence weight $w(E^-)=\frac{s!}{s}=(s-1)!$. The second is the positive edge $E^+=\{(y,\dots,y,t)\}_{t>t_i}$ with $\mu(E^+)=s!$, outgoing primitive vector $(0,\dots,0,1)$ hence weight $w(E^+)=s!$. Only these two edges contribute to balancing in the $\R$-direction, and their contributions $(s-1)!\cdot -s +s!\cdot1=0$ already cancel. For balancing along the $C^s$-direction, $E^+$ does not contribute, so the cancellation must occur between $E^-$ and $L(y)$. 
	Recall that by construction,
	$$L(y):=\bigcup_{j=1}^{s}L(y)_j,\qquad L(y)_j:=\{(y,\dots,\underbrace{z}_{j^{th}},\dots,y,t_i)\,:\,z\in e_y\}\subset C^s\times\{t_i\},$$
	so the primitive vectors for each of the $L(y)_j$ are $(0,\dots,\underbrace{-u}_{j^{th}},\dots,0,0).$
	
	$E^-$ contributes $(s-1)!\cdot u$ in each of the $C$-factors, hence for the balancing condition to hold in the $j^{th}$ factor of $C^s$, $L(y)_j$ must carry weight $$w(L(y)_j)=(s-1)!.$$
	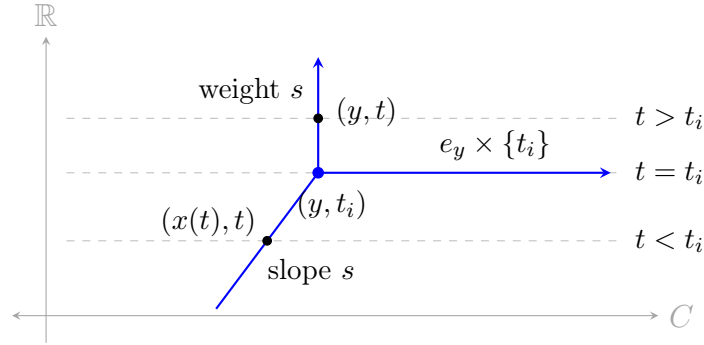
\begin{figure}[h!]
		\centering
		\begin{tikzpicture}[scale=0.9, >=stealth]
			\draw[<->, gray!70] (-0.5,0.9) -- (9,0.9) node[right] {$C$};
			\draw[->, gray!70] (0,0.5) -- (0,5) node[above] {$\R$};
			\draw[dashed, gray!50] (0.3,3.8) -- (8.5,3.8) node[right, black] {\small $t>t_i$};
			\draw[dashed, gray!50] (0.3,3.0) -- (8.5,3.0) node[right, black] {\small $t=t_i$};
			\draw[dashed, gray!50] (0.3,2.0) -- (8.5,2.0) node[right, black] {\small $t<t_i$};
			\draw[thick, blue] (4,3) -- (2.5,1);
			\node[below right] at (3.1,1.9) {\small slope $s$};
			\draw[thick, blue, ->] (4,3) -- (8.3,3);
			\node[above] at (6.6,3.02) {\small $e_y\times\{t_i\}$};
			\draw[thick, blue, ->] (4,3) -- (4,4.7);
			\node[left] at (3.95,4.2) {\small weight $s$};
			\filldraw[blue] (4,3) circle (2.2pt);
			\node[below] at (4.2,2.9) {\small $(y,t_i)$};
			\filldraw (3.25,2) circle (1.8pt);
			\node[above left] at (3.25,1.9) {\small $(x(t),t)$};
			\filldraw (4,3.8) circle (1.8pt);
			\node[right] at (4.1,3.9) {\small $(y,t)$};
		\end{tikzpicture}
		\caption{The local structure of $\graphvarphi$ near $(y,t_i)$, where $\varphi\equiv t_i$ on the segment $e_y\subset C$, and has slope $s$ on the interior edge adjacent to $y$. }
		\label{fig:leaf_edge_local_structure}
	\end{figure}
\end{example}

Otherwise reusing the notation above, let us generalise this example to the setting where $k\geq s$. This means $L(y)$ contains edges adjacent to \textit{multiple} vertices $\tilde{v}:=(v_1,\dots,v_k,t_i)$ of $W$. Recall these edges are indexed by pairs $(\tilde{v},j)$, where $\tilde{v}$ labels the vertex of $W$ to which it is adjacent, and $j$ is a coordinate index of $\tilde{v}$ for which $v_j=y$. The corresponding edge is $$L(y)_{\tilde{v},j}=\big\{(v_1,\dots,v_{j-1},\underbrace{z}_{j^{th}},v_{j+1},\dots,v_k,t_i)\,:\,z\in e_y\big\}\subset C^k\times\{t_i\},$$ with outgoing primitive vector $(0,\dots,\underbrace{-u}_{j^{th}},\dots,0,0)$.

\begin{definition}[Weights for edges $L(y)$]\label{def:weights_of_W_2}
		We define the \emph{multiplicity} of $L(y)_{\tilde{v},j}$ as $$\mu(L(y)_{\tilde{v},j}):=\sum_{E^-}\mu(E^-),$$ where the sum ranges over the negative edges of $W$ adjacent to $\tilde{v}$. Equivalently, $\mu(L(y)_{\tilde{v},j})$ is the signed number of pairs $(\sigma,(i_1,\dots,i_k))$ whose corresponding point converges to $\tilde{v}$ as $t\rightarrow t_i^-$. 
		We set its \emph{weight} to be $$w(L(y)_{\tilde{v},j}):=\frac{\mu(L(y)_{\tilde{v},j})}{s}.$$
\end{definition}

\begin{remark}
	Again, whenever $t_i>\varphi(q)$, $\mu(L(y)_{\tilde{v},j})=w(L(y)_{\tilde{v},j})=0$ for any $(\tilde{v},j)$. 
\end{remark}

\begin{lemma}[Formula for $w(L_{\tilde{v},j})$]
	Assume $t_i<\varphi(q)$. Denote by $\vert y\vert\geq1$ the number of coordinates of $\tilde{v}$ equal to $y$, by $k-k_0$ the number of coordinates equal to $q$, and by $v_1,\dots,v_l$ the distinct values of the remaining coordinates, appearing $n_1,\dots,n_l$ times respectively and whose multiplicities in $A(t)$ are $m_1,\dots,m_l$. Then  $$w(L(y)_{\tilde{v},j})=\frac{(-1)^{k-k_0}}{s}\cdot\frac{s!}{(s-\vert y\vert)!}\cdot(k-k_0)!\cdot\prod_{i=1}^{l}\frac{m_i!}{(m_i-n_i)!}\in\Z.$$
\end{lemma}
\begin{proof}
	As $t\rightarrow t_i^-$, each $y$-coordinate of $\tilde{v}$ is continued by the point $x(t)$ of multiplicity $s$, and each $v_i$-coordinate by its unique continuation of constant multiplicity $m_i$, so that $\tilde{v}$ has a unique adjacent negative edge, and the same count as in Lemma \ref{lemma:closed_formula_multiplicity} yields the result. 
	This is an integer, since $\vert y\vert\geq1$ implies that $\frac{s!}{(s-\vert y\vert)!}$ is divisible by $s$. 
\end{proof}

\begin{remark}
	Note that the weight is independent of the choice of the index $j$ among the $\vert y\vert$ coordinates of $\tilde{v}$ equal to $y$, and that in the situation of Example \ref{example:weights_of_leaf_edges} ($\vert y\vert=k=k_0=s$, $l=0$) it recovers $w=(s-1)!$.
\end{remark}

\begin{proposition}\label{prop:weights_well_defined}
	The weights of Definition \ref{def:weights_of_W_1} are the unique weights on the segments of $\vert W\vert$ for which the fibres of $W$ satisfy Equation \eqref{eq:fibres_of_W}.
\end{proposition}
\begin{proof}
	By Definition \ref{def:fibre_over_R}, the fibre over $x\notin\{t_0,\dots,t_r\}$ of a single edge $E$ of $W$ of weight $w(E)$ is $(\pi_{C^k})_*\operatorname{div}(\varphi_{x}\vert_W)$, where $\phi_{x}$ is the pullback of $\max(t,x)$. 
	Note that for $t\leq x$, the linear part of $\phi_{x}$ is the covector $dt$, and for $t>x$ it is zero. Applying the definition of the divisor of a rational function \cite[Construction~3.3, Definition~3.4]{tropical_intersection_theory_allerman_rau} at the point $E_{t}$, whose two adjacent facets are the two halves of $E$ with primitive outgoing vectors $\pm u_E$, we obtain the coefficient
	\begin{align}
		w(E)\cdot dt(u_E)+w(E)\cdot0= w(E)\cdot dt(u_E)=\mu(E).
	\end{align}
	For $x\in\{t_0,\dots,t_r\}$, Equation \eqref{eq:fibres_of_W} follows from Lemma \ref{lemma:A(t)_well_def}, since the coefficient of a point of $W$ in $\operatorname{div}(\phi_{t_i}\vert_W)$ is the sum of the quantities $w(E)\cdot dt(u_E)$ over its positive adjacent edges. Again, this follows from the fact that the linear part of $\phi_{x}$ is zero for $t<0$, hence on the negative adjacent edges. Uniqueness is immediate. 
\end{proof}

\subsubsection{$W$ is balanced}

We have seen that points with real coordinates in $\R\backslash\{t_0,\dots,t_r\}$ lie on the interior of edges $E$ of $\vert W\vert$, which are furthermore rational for the product tropical structure on $C^k\times\R$, and we have assigned to each such edge an integer weight $w(E)$. To show that this endows $\vert W\vert $ with the structure of a tropical cycle, it remains to prove the balancing condition, which is only non-trivial at vertices of valency $\geq2$.

These vertices are of the form $$\tilde{v}:=(z_{(i_1,\dots,i_k)}(t_i),t_i)\in C^k\times\R$$ for some $i\in\{1,\dots,r\}$, $\sigma\in S_k$, and $\{i_1,\dots,i_k\}\in\{0,1\}^k$. We want to understand the edges adjacent to this vertex, along with their multiplicites/weights. 
Our first step will be to reduce our task to only showing the balancing condition holds along the $k_0$ components of $C^k$ in $C^k\times\R$ which correspond to the indices $i_j=0$.  Let us first verify that the balancing condition always holds along the $\R$-direction:

\begin{lemma}
	Let $\tilde{v}\in C^k\times\R$ be any vertex point of $\vert W\vert$. Then the $\R$-component of the weighted sum of all outgoing primitive vectors for edges adjacent to $\tilde{v}$ is zero.
	\label{lemma:balancing_for_R_component}
\end{lemma}
\begin{proof}
	Edges of $\vert W\vert$ adjacent to $\tilde{v}$ split into negative edges, positive edges, and null edges belonging to $L(y)$ for endpoints $y$ with $\varphi(y)=t_i$. The latter have vanishing $\R$-component by construction, and do not contribute. A positive (resp.\ negative) adjacent edge $E$ has outgoing primitive vector $u_E$ (resp.\ $-u_E$), whose $\R$-component is $c(E)$ (resp.\ $-c(E)$) by Definition 3.16. Since $w(E)\cdot c(E)=\mu(E)$ by Definition 3.17, the $\R$-component of the weighted sum of outgoing primitive vectors is
	$$\sum_{E^+}\mu(E^+)-\sum_{E^-}\mu(E^-),$$
	where the sums range over the positive and negative adjacent edges respectively. 
	For $t_i>\varphi(q)$, we know in fact that both terms are zero, and for $t_i<\varphi(q)$, it follows immediately from Remark \ref{rmk:sign_factors_out} and Lemma \ref{lemma:A(t)_well_def}. It remains to consider the case $t_i=\varphi(q)$. In this case, the multiplicities of any positive adjacent edges are zero, so the only contribution is from multiplicities of negative edges. 
	Recall that the multiplicity of each negative adjacent edge $E^-$ is given by summing $(-1)^{k-k_0}$ over all pairs $(\sigma,\{i_1,\dots,i_i\})$ such that $(z_{\{i_1,\dots,i_k\}}(t),t)\rightarrow \tilde{v}$ when $t\rightarrow t_i^-$. 
	We will show that the contributions from \textit{different} edges cancel out amongst each other.
	 
	Let $n\geq1$ be the number of coordinates of $\tilde{v}$ equal to $q$, and $m$ the number of times $q$ appears in $A(t_i)$. The pairs counted by the multiplicity fill either of the $n$ $q$-coordinates of $\tilde{v}$ with either one of the $m$ points of $A(t_i)$ equal to $q$, or with the corresponding fixed $q$-coordinate (i.e. $i_j=1$). These occur with opposite signs. The negative edges $E^-$ can be seperated into groups indexed by $l\in\{0,\dots,n\}$, where $l$ of the $n$ coordinates of $E^-$ converging to $q$ are non-constant (i.e. i.e. points of $A(t)$ converging to $q$), and the $(n-l)$ others are contant.
	Recording contributions from each of these groups as in Lemma \ref{lemma:closed_formula_multiplicity}, one finds
	\begin{align}
		\sum_{E^-}\mu(E^-)&=\sum_{l=0}^n\binom{n}{l}\frac{m!}{(m-l)!}(-1)^{n-l}(n-l)! \notag
		&= n!(-1)^n\sum_{l=0}^n\binom{m}{l}(-1)^l\notag
	\end{align}
	Now notice that $m=n$: indeed, for any $j\in\{1,\dots,k\}$ for which $x_j(t_i)=q$ is \textit{not} in $A(t_i)$ we must have $i_j=1$, also contributing a $q$-coordinate to $\tilde{v}$. This means the binomial sum is a total sum and is equal to $(1-1)^0=0$, and we conclude. 
	
\end{proof}

Now, we verify that the balancing condition holds along the $C^{k-k_0}$-component of $C^k$ which corresponds to the $j$-coordinates with $i_j=1$:

\begin{lemma}
	Let $\tilde{v}\in C^k\times\R$ be a vertex point of $\vert W\vert$. Then the $(k-k_0)$ $C$-components of the weighted sum of all outgoing primitive vectors for edges adjacent to $\tilde{v}$ corresponding to $i_j=1$ are zero.
	\label{lemma:balancing_for_q_components}
\end{lemma}
\begin{proof}
	Fix an index $j$ with $i_j=1$, so that the $j^{th}$ coordinate of $\tilde{v}$ is $q$. Recall from Lemma \ref{non-zero_slopes} that the endpoints $y$ for $\varphi$ are interior edge points of $C$, and that this choice may be made so that no endpoint equals $q$. Any edge of some $L(y)$ adjacent to $\tilde{v}$ has its unique non-constant coordinate in a slot $j'$ with $v_{j'}=y\neq q$, so that $j'\neq j$ and its outgoing primitive vector has vanishing $j^{th}$ component. So we restrict to \textit{monotone} (i.e.positive or negative) adjacent edges. 
	Recall that if $t_i>\varphi(q)$, the multiplicity of \textit{any} adjacent edge is zero, hence balancing is trivial. 
	
	If $t_i<\varphi(q)$, every monotone adjacent edge $E$ has $i_j=1$, hence the $j^{th}$ coordinate is constant equal to $q$ and the $j^{th}$ component of the primitive outgoing vector is zero.
	
	It remains to consider the case $t_i=\varphi(q)$. Positive adjacent edges have interior $\R$-coordinates in $(\varphi(q),t_{i+1})$, hence weight zero by Remark \ref{rmk:sign_factors_out}, so only negative adjacent edges contribute. Denote by $e_1,\dots,e_d$ the segments of $C$ adjacent to $q$ on which $\varphi<\varphi(q)$, by $s_\alpha$ the absolute value of the slope of $\varphi$ on $e_\alpha$, and by $u_\alpha$ the outgoing primitive vector of $e_\alpha$ from $q$. For $t\in(t_{i-1},t_i)$, points of $A(t)$ converging to $q$ are points $x_\alpha(t)$ in the interior of each $e_\alpha$, of multiplicity $s_\alpha$. We set $s:=s_1+\dots+s_d$, the multiplicity of $q$ in $A(t_i)$, and $n$ to be the number of coordinates of $\tilde{v}$ equal to $q$. Notice that $n\geq s$. Indeed, $s$ indices satisfy $x_j(t_i)=q$; if $r$ of these have $i_j=1$, then $n=(k-k_0)+(s-r)\geq s$, since $r\leq k-k_0$.

	A negative adjacent edge $E$ contributes to the $j^{th}$ component of the balancing condition only if its $j^{th}$ coordinate is non-constant, in which case this coordinate is the point $x_\alpha(t)$ for some $\alpha\in\{1,\dots,d\}$. By Definitions \ref{def:multiplicity_of_segment} and \ref{def:weights_of_W_1}, it contributes  $$w(E)\cdot\frac{c(E)}{s_\alpha}\,u_\alpha=\frac{\mu(E)}{s_\alpha}\,u_\alpha.$$ 
	
	Let $N_\alpha$ be the signed number of pairs $(\sigma,\{i_1,\dots,i_k\})$ such that $\sigma(z_{\{i_1,\dots,i_k\}(t)},t)$ whose $j^{th}$ coordinate is a label of $x_\alpha(t)$, and which converges to $\tilde{v}$ as $t\rightarrow t^-_i$. Then the $j^{th}$ component of the balancing condition is $$\sum_{\alpha=1}^{d}\frac{u_\alpha}{s_\alpha}\,N_\alpha.$$

	We claim each $N_\alpha$ vanishes. Every pair contributing to $N_\alpha$ consists of the following choices. The first is a matching of the coordinates of $\tilde{v}$ different from $q$ with labels of their continuations below $t_i$, in some number $M\geq0$ of ways. The second is the choice of one of the $s_\alpha$ labels of $x_\alpha(t)$ for the slot $j$. The third is the choice of $l\leq n-1$ coordinates of $A(t)\backslash\{x_\alpha(t)\}$ which converge to $q$ (i.e. of type $x_{\alpha'}(t))$. Notice that the stronger bound of $l\leq s-1$ is forced. The fourth is a free distribution of the $(n-1-l)$ remaining $q$-slots of $\tilde{v}$, which inherit $i_{j'}=1$ and contribute $-1$ signs. Hence
	\begin{align}
		N_\alpha&=M\cdot s_\alpha\cdot\sum_{l=0}^{s-1}\binom{n-1}{l}\cdot\frac{(s-1)!}{(s-1-l)!}\cdot(-1)^{n-1-l}\cdot(n-1-l)! \notag \\
		&=M\cdot s_\alpha\cdot(n-1)!\cdot(-1)^{n-1}\sum_{l=0}^{s-1}(-1)^{l}\binom{s-1}{l} \notag \\
		&= M\cdot s_\alpha\cdot(n-1)!\cdot(-1)^{n-1}\cdot(1-1)^{s-1} \notag \\
		&=0, \notag
	\end{align}
	where the passage from the first to the second line uses $\binom{n-1}{l}\cdot(n-1-l)!=\frac{(n-1)!}{l!},$ and the final cancellation requires that $s-1\geq 1$, which we now prove.
	Write $\rho_\beta$ for the slopes of the segments of $C$ adjacent to $q$ on which $\varphi>\varphi(q)$. Then the coefficient of $q$ in $\mathrm{div}(\varphi)$ is $\sum_\beta\rho_\beta-s\leq-1$. If there is at least one such segment, then $s\geq1+\sum_\beta\rho_\beta\geq2$. Otherwise, all $d\geq2$ segments of $C$ adjacent to $q$ satisfy $\varphi<\varphi(q)$, and $s\geq d\geq2$.
\end{proof}

We have reduced the proof of the balancing condition for $\vert W\vert$ to the remaining $k_0$ factors of $C^k\times\R$; we now carry this out.

\begin{proposition}\label{prop:W_balanced}
	Endowed with the weights of Definitions \ref{def:weights_of_W_1} and \ref{def:weights_of_W_2}, the weighted polyhedral complex $\vert W\vert$ satisfies the balancing condition at each of its vertices. In particular, $W$ is a tropical $1$-cycle in $C^k\times\R$.
\end{proposition}

\begin{proof}
	Let $\tilde{v}=(z_{(i_1,\dots,i_k)}(t_i),t_i)$ be a vertex of $\vert W\vert$ (as before we assume $\sigma=\mathrm{id}$). By Lemma \ref{lemma:balancing_for_R_component}, the $\R$-component of the weighted sum of outgoing primitive vectors vanishes, and by Lemma \ref{lemma:balancing_for_q_components}, so do the $C$-components at the indices $j$ with $i_j=1$. Fix an index $j$ with $i_j=0$, and set $v_j:=x_j(t_i)\in C$.
	
	As usual, for $t_i>\varphi(q)$, all adjacent edges to $\tilde{v}$ have weight zero and there is nothing to prove. We will treat the cases $t_i=\varphi(q)$ and $t_i<\varphi(q)$ separately, but we first introduce some shared notation.
	Let $E$ be a non-null edge adjacent to $\tilde{v}$. If the $j^{th}$ coordinate of $E$ is constant, i.e. coming from a semi-infinite edge of $\graphvarphi$, $E$ contributes nothing. Otherwise, the $j^{th}$ coordinate of $E$ travels along a segment $e$ of $C$ adjacent to $v_j$, and the corresponding point appears in $A(t)$ with multiplicity $s_e$ equal to the absolute value of the slope of $\varphi$ on $e$. Denoting by $u_e$ the primitive direction at $v_j$ pointing into $e$, the $j^{th}$ component of the outgoing primitive vector of $E$ is $\frac{c(E)}{s_e}u_e$ by Definition \ref{def:primitive_vector}. We set $N_e$ to be the signed number of pairs $(\sigma,\{i_1,\dots,i_k\})$ such that $\sigma(z_{i_1,\dots,i_k}(t),t)$ converges to $\tilde{v}$ as $t\rightarrow t_i$, and whose $j^{th}$ coordinate travels along $e$, i.e. $$N_e:=\sum_{\substack{E \text{ non-null adjacent to $\tilde{v}$},\\ j^{th}\text{ coordinate on } e}}\mu(E).$$
	Since $w(E)\cdot c(E)=\mu(E)$, the $j^{th}$ component of the contributions from non-null edges adjacent to $\tilde{v}$ to the balancing condition is
	$$\sum_{e\ni v_j}\frac{u_e}{s_e}\,N_e,$$
	where $e$ ranges over the segments of $C$ adjacent to $v_j$ on which $\varphi$ is non-constant.

	If $t_i=\varphi(q)$, we claim $N_e=0$ for every $e$. If $v_j=q$, this is established in the proof of Lemma \ref{lemma:balancing_for_q_components}. If $v_j\neq q$, a very similar reasoning applies, except that $j$ no longer indexes a point converging to $q$. Again, if $n$ is the number of coordinates of $\tilde{v}$ equal to $q$, and $s$ the number of indiced of $A(t)$ labelling points which converge to $q$ when $t\rightarrow t_i^-$ (so that $n\geq s\geq2$, as in the proof of Lemma \ref{lemma:balancing_for_q_components}), $N_e$ can be written as a sum
	$$\sum_{l=0}^{s}\binom{n}{l}\cdot\frac{s!}{(s-l)!}\cdot(-1)^{n-l}(n-l)!=n!\,(-1)^{n}\sum_{l=0}^{s}(-1)^{l}\binom{s}{l}=n!\,(-1)^n(1-1)^{s}=0,$$ where $l$ indexes the number of non-constant coordinates of an edge $E$ which converge to $q$. Finally, assume there is a null edge adjacent to $\tilde{v}$, which is equivalent to the existence of an endpoint $y$ for $\varphi$ for which $\varphi(y)=\varphi(q)=t_i$. By Definition \ref{def:weights_of_W_2}, any edge of $L(y)$ adjacent to $\tilde{v}$ has weight $\frac{1}{s'}\sum_{E^-}\mu(E^-)$, where $s'$ is the non-zero slope of $\varphi$ on the adjacent edge of $y$. We have just shown that this sum vanishes; so the $j^{th}$-component of the balancing condition is satisfied.

	Now we treat the final case $t_i<\varphi(q)$. We distinguish three sub-cases based on the type of point $(v_j,t_i)$ of $\graphvarphi$. 

\emph{First case:} when $(v_j,t_i)$ lies in the interior of an edge of $\graphvarphi$. Note this edge could be the image of an edge of $C$, or a semi-infinite edge coming from the balanced graph construction. There are exactly two segments $e^-$ and $e^+$ of $C$ adjacent to $v_j$, in the direction of increasing and decreasing $t_i$, so $s_{e^-}=s_{e^+}=\vert s_j\vert$ and $u_{e^-}=-u_{e^+}$. Every edge of $\vert W\vert$ converging to $\tilde{v}$ from below (resp. above) has its $j^{th}$ coordinate on $e^-$ (resp.\ $e^+$). The signed count of these negative and positive edges coincide following Lemma \ref{lemma:A(t)_well_def} and Remark \ref{rmk:sign_factors_out}. Denoting this count $\mu$, the $j^{th}$-component of the balancing condition is $\frac{\mu}{\vert s_j\vert}(u_{e^-}+u_{e^+})=0$.

\emph{Second case:} when $(v_j,t_i)$ is a vertex of $\graphvarphi$ of valency $\geq3$ which is not of the form $(y,\varphi(y))$ for an endpoint $y$. Let $n(v_j)$ be the number of coordinates of $\tilde{v}$ equal to $v_j$, and let $m(v_j)$ be the multiplicity of $v_j$ in $A(t_i)$.

	 For any segment $e$ of $C$ adjacent to $v_j$, recall that whether $\varphi(e)\leq\varphi(v_j)$ or $\varphi(e)\geq\varphi(v_j)$ determines whether the edges counted in $N_e$ have $\R$-coordinates $\leq t_i$ of $\geq t_i$.
	 A pair counted by $N_e$ consists of the following data. First, one of the $s_e$ labels on $e$ for the slot $j$. Second, a choice of $n(v_j)-1$ other indices among the $m(v_j)-1$ which correspond to points of $A(t)$ converging to $v_j$ as $t\rightarrow t_i$, in $\frac{(m(v_j)-1)!}{(m(v_j)-n(v_j))!}$ ways. Finally, a continuation of all remaining coordinates of $\tilde{v}$, in some number $M$ of ways. It is not hard to see that $M$ does not depend on $e$.

	Hence $$N_e=(-1)^{k-k_0}\cdot M\cdot  m_e\,\frac{(m(v_j)-1)!}{(m(v_j)-n(v_j))!}$$ for \emph{any} adjacent segment $e$, and the $j^{th}$ component equals
	$$(-1)^{k-k_0}\cdot M\frac{(m(v_j)-1)!}{(m(v_j)-n(v_j))!}\sum_{e\ni v_j}u_e=0,$$
	where the sum ranges over all segments of $C$ adjacent to $v_j$ ($\varphi$ is non-constant on each of them, by Proposition 3.3 and since $v_j$ is not an endpoint), and vanishes by the balancing condition on $C$.
	
	\emph{Third case:} when $v_j=y$ is an endpoint for $\varphi$ with $\varphi(y)=t_i$. Then $y$ is $2$-valent, and none of the points on its adjacent segment $e_y$ on which $\varphi$ is constant are in $A(t)$ by construction. Let $s$ be the non-zero slope of $\varphi$ on the other (interior) segment $e$ adjacent to $y$, pointing away from $y$. Let us assume $s<0$, the case $s>0$ is treated identically. Then $\graphvarphi$ has a semi-infinite edge of the form $\{y\}\times[t_i;+\infty)$ of weight $s$. Any positive edge adjacent to $\tilde{v}$ therefore has constant $j^{th}$-coordinate equal to $y$, hence does not contribute to balancing in this direction. Any negative edge $E$ has $j^{th}$-coordinate on $e$, hence the $j^{th}$ component of its primitive vector is $\frac{c(E)}{s}u_e$ and $E$ contributes $w(E)\cdot\frac{c(E)}{s}u_e=\frac{\mu(E)c(E)}{c(E)s}u_e=\frac{\mu(E)}{s}u_e$ to the $j^{th}$-component of the balancing condition. By Definition \ref{def:weights_of_W_2}, the sum of contributions from each of these edges is therefore $\frac{\mu(L(y)_{\tilde{v},j})}{s}u_e$.
	Recall from Definition \ref{def:weights_of_W_2} that the weight of the edge $L(y)_{(\tilde{v},j)}$ is $\frac{\mu(L(y)_{\tilde{v},j})}{s}$ and its primitive vector is $-u_e$ in the $j^{th}$-component and zero elsewhere. Hence balancing along the $j$-direction yields $$\left(\frac{\mu(L(y)_{\tilde{v},j})}{s}- \frac{\mu(L(y)_{\tilde{v},j})}{s}\right) u_e=0.$$
	 
	In all cases, the $j^{th}$ component of the weighted sum vanishes, from which we conclude.
\end{proof}

We have thus completed the proof of Proposition \ref{prop:tropical_voevodsky_curves}, and therefore of Theorem \ref{thm:tropical_voevodsky}.

\section{Applications}\label{section:examples}

\subsection{Tropical product $2$-tori}\label{section:examples_1}

In this section, we carry out explicitly the main construction of the rational equivalence $W\subset C^k\times\R$ from Proposition \ref{prop:tropical_voevodsky_curves}, in the case where $C\cong S^1$ has genus $1$. Notice that the data of a tropical structure on $S^1\cong \R/\Z$ is simply the data of a lattice in $\R$, or, equivalently, of any non-zero real number $\frac{1}{a}\in\R$. From the metric graph perspective, this amounts to setting the diameter of $S^1$ to be $a$. We denote the corresponding tropical curve by $S^1_a$. We prove the following:

\begin{proposition}\label{prop:CH_0_for_T^2}
	For any $a\in\R$ and two points $p,q\in S^1_a$, $2\cdot(p-q)^2=0\in CH_0((S_a^1)^2).$
\end{proposition}

\begin{remark}
	Following Remark \ref{remark:increasing_k}, this matches our expectation that the optimal exponent $k$ for which Proposition \ref{prop:tropical_voevodsky_curves} holds should be $g(C)+1$. However, it is \textit{not} the case that we can construct a rational function $\varphi$ on $S^1_a$ with divisor $2p-q-y$ for some $y\in C$ whose slopes are non-zero; see Example \ref{example:circle_must_have_zero_slopes} below. To find a rational function with nowhere-zero slopes whose divisor is $k\cdot p-q-y_1-\dots-y_{k-1}$, we need to take $k\geq3$, as we show in Example \ref{ex:increase_k_get_non-zero_slopes}. In that case, our construction would follow through and yield $3!\cdot(p-q)^3\sim0$. However, because the case $C\sim S^1$ is simple enough, we are able to deal with zero slopes manually, and prove the optimal result with $k=2$. 
\end{remark}

\begin{example}\label{example:circle_must_have_zero_slopes}
	Identify $S^1_a$ with $\mathbb{R}/a\mathbb{Z}$, with $p=0$ and $q$ at distance
	$l_1$ from $p$. A rational function $\varphi$ with
	$\operatorname{div}(\varphi)=2p-q-y$ must be affine on each of the three arcs
	cut out by $p$, $q$ and $y$. If $s_1,s_2,s_3$ denote its slopes on the arcs
	$p\to q$, $q\to y$, $y\to p$ (of lengths $l_1,l_2,l_3$), the divisor
	condition forces $s_2=s_1-1$ and $s_3=s_1-2$, while continuity around the
	circle gives $s_1\cdot a=l_2+2l_3$. As $0<l_2+2l_3<2a$, we must have
	$s_1=1$, so $\varphi$ is unique up to an additive constant and its slopes are
	$(1,0,-1)$. The only possibility is that $y$ is the reflection of $q$ through $p$ (equivalently,
	$y=2p-q$ in the Jacobian $J(S^1_a)\cong\mathbb{R}/a\mathbb{Z}$).
	In particular $\varphi$ has slope zero on the arc from $q$ to $y$; so the only case in which our construction from Section \ref{section:results} runs through is if $q=y$ is antipodal to $p$, i.e.
	$d=a/2$, in which case $\operatorname{div}(\varphi)=2p-2q$ with slopes
	$(1,-1)$.
%
%
\end{example}

\begin{example}\label{ex:increase_k_get_non-zero_slopes}
	Using $S^1_a\cong\mathbb{R}/a\mathbb{Z}$ as above, with $p=0$ and $q$ at distance
	$l_1$ from $p$. This time, we are looking for a rational function $f$ on $S_a^1$ with
	$\mathrm{div}(f)=3\cdot p-q-y_1-y_2$ whose slopes are nowhere zero. Set $l_1$, $l_2$,
	$l_3$ and $l_4$ to be the lengths of the segments $p\rightarrow q$, $q\rightarrow y_1$,
	$y_1\rightarrow y_2$, $y_2\rightarrow p$, and $s_1,s_2,s_3,s_4$ the slopes of $f$ on
	these respective segments. Again, the divisor condition imposes $s_2=s_1-1$,
	$s_3=s_2-1$, and $s_4=s_3-1=s_1-3$. Continuity of $f$ imposes
	$s_1\cdot a=l_2+2l_3+3l_4.$ As the right-hand side lies strictly between $0$ and $3a$, the only integer values
	available are $s_1=1$ or $s_1=2$ which yield slopes $(1,0,-1,-2)$ and $s_1=2$
	gives $(2,1,0,-1)$. More generally, continuity imposes that the slopes take both positive and negative values on different segments of $C$. If all points of the divisor have coefficients $\pm1$, slopes will \enquote{jump} by $\pm1$, and therefore necessarily go through zero while going from positive to negative values. This shows no choice of four \emph{distinct}
	points $p,q,y_1,y_2$ yields nowhere-zero slopes.
	We therefore set $y_1=y_2=:y$. Taking $s_1=2$, we have collapsed the segment on which $f$ is constant and we now get slopes $(2,1,-1)$ on the arcs $p\rightarrow q$,
	$q\rightarrow y$, $y\rightarrow p$. Working out the continuity condition reveals collapse works when $l_1<\frac{a}{3}$; otherwise the collapse $q=y_1$ works instead. Assuming this, taking the lengths of $q\rightarrow y$ to be $\frac{a-3l_1}{2}$ yields the desired function. Its two-ended balanced graph is represented in Figure \ref{fig:two_ended_balanced_graph_full_example}, and our construction can be applied to it to show that $3!\cdot(p-q)^3\sim0$ in $(S^1_a)^3$ for any $p,q\in S^1_a$. 

\begin{figure}[h!]
	\centering
	\begin{minipage}{0.45\textwidth}
		\centering
		\begin{tikzpicture}[scale=1.1]
			\draw[thick] (0,0) circle (1);
			
			\fill (90:1)   circle (1.2pt) node[above=2pt] {$p$};
			\fill (15:1)   circle (1.2pt) node[right=2pt] {$q$};
			\fill[red] (285:1) circle (1.2pt) node[below right=1pt, red] {$y_1$};
			\fill[red] (195:1) circle (1.2pt) node[below left=1pt, red] {$y_2$};
			
			
		\end{tikzpicture}
	\end{minipage}
	\hfill
	\begin{minipage}{0.45\textwidth}
		\centering
		\begin{tikzpicture}[scale=1.1]
			\draw[thick] (0,0) circle (1);
			
			\fill (90:1)  circle (1.2pt) node[above=2pt] {$p$};
			\fill (30:1)  circle (1.2pt) node[above right=1pt] {$q$};
			\fill[red] (300:1) circle (1.2pt) node[below right=1pt, red] {$y$};
			
			\node[blue] at (60:0.78)  {$2$};
			\node[blue] at (-15:0.78) {$1$};
			\node[blue] at (195:0.78) {$-1$};
			
		\end{tikzpicture}
	\end{minipage}
	\caption{Left: four distinct points, where continuity forces a vanishing slope. Right: the
		collision $y_1=y_2=y$ with $q$ at distance $l_1<a/3$ from $p$, for which
		$\mathrm{div}(f)=3p-2y-q$ is realised by a function with slopes $(2,1,-1)$.}
	\label{fig:collapse_get_nonzero_slopes}
\end{figure}
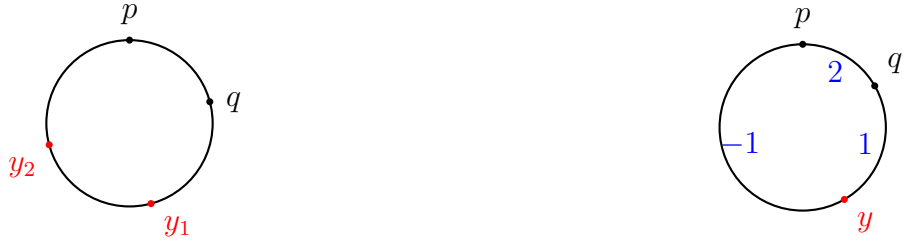
\end{example}

\begin{proof}[Proof of Proposition \ref{prop:CH_0_for_T^2}]
	We use the rational function $\varphi$ from Example \ref{example:circle_must_have_zero_slopes} to construct a $1$-cycle $W\subset (S_a^1)\times\R$ realising the rational equivalence $2\cdot(p-q)^2\sim0$. Because $\varphi$ is only well-defined up to a constant, we set $\varphi(p)=0$. 
	The minimum of $\varphi$ is $\varphi(p)=0$, and its maximum is $\varphi(q)=\varphi(y)=l$. Defining $A(t)$ as in Section \ref{section:results} with the subtlety that we assign multiplicity zero to points on the interior of the edge on which $\varphi$ is constant, we get 
	\begin{align}
	A(t)=\begin{cases}
		\{p,p\}& \text{for $t\leq0$} \\
		\{x_1(t),x_2(t)\} & \text{for $t\in[0,l]$}\\
		\{q,y\} & \text{for $t\geq l$},
	\end{cases} \notag
\end{align}
	where $x_1(t)$ and $x_2(t)$ are given by $t$ and $-t$ respectively in the parametrisation $S^1_a\sim\mathbb{R}/a\mathbb{Z}$ where $p=0$, $q=l$ and $y=-l$. 
%
	The rational equivalence in $W$ has edges of type $N$ with $\R$-coordinates in $(-\infty,0]$, of type $F$ with $\R$-coordinates $[0,l]$, of type $P$ with $\R$-coordinates in $[l;+\infty)$, and of type $L$ with fixed $\R$-coordinate $l$. 
	The edges of type $N$ are
	\begin{align}
		& N_1=\{(p,p)\}\times(-\infty;0] & &N_2=\{(p,q)\}\times(-\infty;0]& \notag\\
		&N_3=\{(q,p)\}\times(-\infty;0]& &N_4=\{(q,q)\}\times(-\infty;0]& \notag
	\end{align}
	with weights $w(N_1)=w(N_4)=2$ and $w(N_2)=w(N_3)=-2$.
	The edges of type $F$ are
\begin{center}
	\makebox[\linewidth]{\hfill$F_1=\{(x_1(t),x_2(t),t)\}_{t\in[0,l]}$\hfill$F_2=\{(x_2(t),x_1(t),t)\}_{t\in[0,l]}$\hfill}\\[4pt]
	\makebox[\linewidth]{\hfill$F_3=\{(x_1(t),q,t)\}_{t\in[0,l]}$\hfill$F_4=\{(q,x_2(t),t)\}_{t\in[0,l]}$\hfill}\\[4pt]
	\makebox[\linewidth]{\hfill$F_5=\{(q,x_1(t),t)\}_{t\in[0,l]}$\hfill$F_6=\{(x_2(t),q,t)\}_{t\in[0,l]}$\hfill$F_7=\{(q,q,t)\}_{t\in[0,l]}$\hfill}
\end{center}
	with weights $w(F_1)=w(F_2)=1$, $w(F_7)=2$, and $w(F_3)=w(F_4)=w(F_5)=w(F_6)=-1$ (because all slopes have absolute value $1$, the weights coincides with the multiplicities). 
	There are three of type $P$ with fixed $(S_a^1)^2$-coordinates equal to $\{(q,y)\}$, $\{y,q\}$ and $\{q,q\}$. Because as we have seen, for combinatorial reasons they appear with multiplicity/weight zero we will not take them into account. 
	To add the appropriate edges of type $L$, we need to understand the balancing condition at each of the vertices. An explicit verification yields that the balancing condition holds for all four vertices at $t=0$, as expected. It remains to check the three vertices $(q,q,l)$, $(y,q,l)$ and $(q,y,l)$ at $t=l$. 
	For the vertex $(q,q,l)$ the adjacent vertices are $F_3$, $F_5$ and $F_7$, and their primitive tangent vectors are respectively $(0,-1,-1)$, $(-1,0,-1)$ and $(0,0,1)$. Their weighted sum is therefore $(-1,-1,0)$. We add the following two edges, each with weight $-1$ (which can be understood from the presence of a constant $q$-factor):
	\begin{align}
		&L_1=[q,y]\times\{y\}\times\{l\}&&L_2=\{q\}\times[q,y]\times\{l\},& \notag
	\end{align}
	where $[q,y]\cong[l,a-l]$ in the parametrisation $S_a^1\cong\R/a\Z$ is the segment from $q$ to $y$ not containing $p$ (see the figure from Example \ref{example:circle_must_have_zero_slopes}). This makes $(q,q,l)$ balanced. 
	For the vertex $(y,q,l)$, the two adjacent edges were $F_2$ and $F_6$, with primitive vectors $(1,-1,-1)$ and $(1,0,-1)$. It is also adjacent to $L_1$ with primitive vector $(-1,0,0)$. Their weighted sum is thus $(1,-1,0)$.
	Similarly, the vertex $(q,y,l)$ is adjacent to $F_1$, $F_4$ and $L_2$, with primitive vectors $(-1,1,-1)$, $(0,1,-1)$ and $(0,-1,0)$. Their weighted sum is thus $(-1,1,0)$. 
	We insert the edge $$\{(z(s),z(1-s),l)\}_{s\in[0,1]},$$ where $z(s)$ is the constant speed path from $q$ to $y$ which doesn't pass through $p$, given in the parametrisation by $z(t)=(1-s)\cdot l+s\cdot (a-l)$. Its endpoints are $(q,y,l)$ and $(y,q,l)$, and its primitive tangent vector from $(q,y,l)$ is $(1,-1,0)$ and from $(y,q,l)$ is $(-1,1,0)$. Therefore, assigning it weight $w(L_3)=1$ makes both $(q,y,l)$ and $(y,l,q)$ balanced. 

	We have thus constructed a $1$-cycle in $(S^1_a)^2\times\R$, with $W_t=2\cdot(p-q)^2$ for $t<0$ and $W_t=0$ for any $t>l$, which yields the desired result. 
\end{proof}

\subsection{Tropical Ceresa cycles of hyperelliptic curves}\label{section:examples_2}

In this Section, we apply our Theorem \ref{thm:tropical_voevodsky} to prove that the Ceresa cycle of a hyperelliptic tropical curve $C$ is smash-nilpotent in its Jacobian $J(C)$. 

Recall that the tropical Jacobian of $C$ is the torus
\begin{align}
	J(C):=H^0(T_\R^*C)^*/H_1(C;\Z),	\notag
\end{align}
where $H_1(C;\Z)$ is identified with a lattice inside $H^0(T_\R^*C)^*$ by setting $\gamma(\alpha):=\int_\gamma\alpha$ for any $\alpha\in H^0(T_\R^*C)$. If the underlying graph of $C$ has genus $g$, then $H^0(T_\Z^*C)$ is a free abelian group of rank $g$, so that $J(C)$ is a $g$-dimensional torus.

Given a choice of basepoint $p\in C$, the tropical Abel-Jacobi map
\begin{align}
	AJ_p:\hspace{2mm}& C\longrightarrow J(C) \notag \\
	& q\mapsto \left(\alpha\in\Omega^1(C)\mapsto \int_{\gamma:p\rightarrow q}\alpha\right) \notag
\end{align}
embeds $C$ into $J(C)$, where $\gamma:p\rightarrow q$ is any path in $C$ between $p$ and $q$.

A simple example in which $C$ is a genus $2$ curve is represented in Figure \ref{fig:genus_2_Jacobian}.

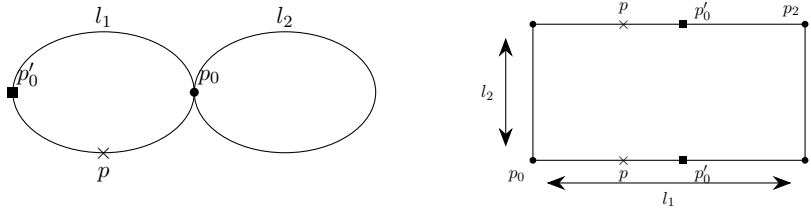
\begin{figure}[htbp]
	\centering
	\begin{tikzpicture}[
		scale=0.8,
		every node/.style={transform shape},
		dot/.style      = {circle, fill, inner sep=1.5pt},
		sq/.style       = {rectangle, fill, inner sep=2.5pt, minimum size=5pt},
		cross/.style    = {font=\small},
		lbl/.style      = {font=\small},
		]
		
		\coordinate (p0) at (0,0);
		\draw (-1.5,0) ellipse [x radius=1.5, y radius=1];
		\draw (1.5,0)  ellipse [x radius=1.5, y radius=1];
		\node[dot] at (p0) {};
		\node[lbl] at (0.25,0.25) {$p_0$};
		\node[sq] (p0prime) at (-3,0) {};
		\node[lbl] at (-2.75,0.3) {$p_0'$};
		\node[cross] (xmark) at (-1.5,-1) {$\times$};
		\node[lbl] at (-1.5,-1.35) {$p$};
		\node[lbl] at (-1.5,1.25) {$l_1$};
		\node[lbl] at (1.5,1.25)  {$l_2$};
		
		\begin{scope}[xshift=5.6cm, yshift=-1.125cm, scale=0.75]
			
			\coordinate (A) at (0,0);
			\coordinate (B) at (6,0);
			\coordinate (C) at (6,3);
			\coordinate (D) at (0,3);
			
			\draw (A) -- (B);
			\draw (D) -- (C);
			\draw (A) -- (D);
			\draw (B) -- (C);
			
			\node[dot] at (A) {};
			\node[dot] at (B) {};
			\node[dot] at (C) {};
			\node[dot] at (D) {};
			
			\node[lbl] at (-0.35,-0.35) {$p_0$};
			
			\node[cross] at (2,3)   {$\times$};
			\node[lbl]   at (2,3.35) {$p$};
			
			\node[sq] at (3.3,3) {};
			\node[lbl] at (3.75,3.3) {$p_0'$};
			
			\node[lbl] at (5.7,3.3) {$p_2$};
			
			\node[cross] at (2,0)   {$\times$};
			\node[lbl]   at (2,-0.35) {$p$};
			
			\node[sq] at (3.3,0) {};
			\node[lbl] at (3.75,-0.3) {$p_0'$};
			
			\draw[{Stealth[length=2.5mm]}-{Stealth[length=2.5mm]}] (5.7,-0.5) -- (0.3,-0.5);
			\node[lbl] at (3,-0.85) {$l_1$};
			
			\draw[{Stealth[length=2.5mm]}-{Stealth[length=2.5mm]}] (-0.6,0.3) -- (-0.6,2.7);
			\node[lbl] at (-1.0,1.5) {$l_2$};
			
		\end{scope}
		
	\end{tikzpicture}
	\caption{A genus $2$ curve on the left and its Jacobian on the right, with Weierstrass points $p_0$ and $p'_0$, and non-Weierstrass point $p$. The pairs of vertical and horizontal are identified, making $J(C)$ is a $2$-torus. }
	\label{fig:genus_2_Jacobian}
\end{figure}

Furthermore, $J(C)$ inherits a group structure from the one on $H^0(T_\R^*C)^*$, hence carries a natural involution $$(-1):J(C)\longrightarrow J(C)$$ obtained by taking inverses. The \textit{Ceresa cycle} of $C$ via $p$ is
\begin{equation}
	AJ_p(C)-(-1)_*(AJ_p(C))\in CH_1(J(C)). \notag
\end{equation}
We will abuse notation and write $C_p$ instead of $AJ_p(C)$, $C_p^-$ instead of $(-1)_*(AJ_p(C))$, and hence the Ceresa cycle simply as $C_p-C_p^-$.  

Notice that for different choices of $p\in C$, the cycles $C_p$ are algbraically equivalent $1$-cycles in $J(C)$, witnessed by the algebraic equivalence $\{(p,AJ_p(C)):p\in C\}\subset C\times J(C)$. Similarly, the algebraic equivalence class of the Ceresa cycle $C_p-C_p^-$ is independent of $p$. One can verify that if $C$ is hyperelliptic, choosing $p$ to be a Weierstrass point yields $C_p=C_p^-$, therefore the Ceresa cycle of a hyperelliptic curve is algebraically trivial. 

In the case of the genus $2$ curve depicted above in Figure \ref{fig:genus_2_Jacobian}, we have represented the two components $C_p$ and $C_p^-$ of the tropical Ceresa cycle with basepoint $p\in C$ in Figure \ref{fig:genus_2_Ceresa}. The left-hand (resp. right-hand) side represents $C_p$ (resp. $C^-_p$). 

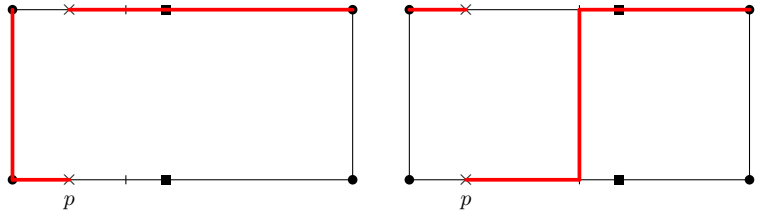
\begin{figure}[htbp]
	\centering
	\begin{tikzpicture}[
		scale=0.75,
		every node/.style={transform shape},
		dot/.style   = {circle, fill, inner sep=1.7pt},
		sq/.style    = {rectangle, fill, inner sep=2.5pt, minimum size=5pt},
		cross/.style = {font=\small},
		lbl/.style   = {font=\small},
		redpath/.style = {red, line width=1.4pt, line cap=round},
		]
		
		\foreach \shift/\xm in {0/0, 7/3} {
			\begin{scope}[xshift=\shift cm]
				
				\coordinate (A) at (0,0); 
				\coordinate (B) at (6,0); 
				\coordinate (C) at (6,3); 
				\coordinate (D) at (0,3); 
				
				\draw (A) -- (B) -- (C) -- (D) -- cycle;
				
				\node[dot] at (A) {};
				\node[dot] at (B) {};
				\node[dot] at (C) {};
				\node[dot] at (D) {};
				
				\node[cross] at (1,0) {$\times$};
				\node[lbl]   at (1,-0.4) {$p$};
				
				\node[cross] at (1,3) {$\times$};
				
				\pgfmathsetmacro{\tA}{2 + \xm/3}
				\pgfmathsetmacro{\tB}{2.7 + \xm/3}
				\draw (\tA,-0.08) -- (\tA,0.08);
				\draw (\tA,2.92) -- (\tA,3.08);
				\node[sq] at (\tB,0) {};
				\node[sq] at (\tB,3) {};
				
				\ifdim \xm pt=0pt
				\draw[redpath] (D) -- (A);
				\draw[redpath] (A) -- (1,0);
				\draw[redpath] (1,3) -- (C);
				\else
				\draw[redpath] (1,0) -- (\xm,0);
				\draw[redpath] (\xm,0) -- (\xm,3);
				\draw[redpath] (\xm,3) -- (C);
				\draw[redpath] (D) -- (1,3);
				\fi
				
			\end{scope}
		}
		
	\end{tikzpicture}
	\caption{The two components of the Ceresa cycle of $C$ via $p$: $C_p\subset J(C)$ on the left, and $C_p^-\subset J(C)$ on the right.}
	\label{fig:genus_2_Ceresa}
\end{figure}

Then our Theorem \ref{thm:tropical_voevodsky} implies the following fact, non-trivial when $p$ is not a Weierstrass point of $C$:
\begin{corollary}
	Let $C_p-C_p^-$ be the tropical Ceresa cycle of $C$ in $J(C)$ for \textit{any} basepoint $p\in C$. Then $C_p-C_p^-$ is smash-nilpotent. 
\end{corollary}

\begin{remark}
	Because the algebraic equivalence $C_p-C_p^-\sim 0$ has base $C$, we expect that $k!\cdot(C_p-C_p^-)^k$ is rationally trivial whenever $k> g(C)$. 
\end{remark}

\vspace{3mm}
\printbibliography
	
\end{document}